\documentclass[12pt, a4paper]{amsart}

\usepackage{geometry} 
\usepackage{fancyhdr}
\usepackage{url}
\usepackage{amssymb}
\usepackage{amsmath}
\usepackage{accents}
\usepackage{eucal}
\usepackage{amscd}
\usepackage[shortalphabetic]{amsrefs}
\usepackage{mathptmx}
\usepackage{times}
\usepackage{graphics}

\usepackage{units}
\allowdisplaybreaks[1]

\newtheorem{theorem}{Theorem}[section]
\newtheorem{lemma}[theorem]{Lemma}

\newtheorem{definition}[theorem]{Definition}
\newtheorem{proposition}[theorem]{Proposition}

\newtheorem{remark}[theorem]{Remark}

\newcommand{\circo}{\accentset{\circ}}

\DeclareMathOperator{\tr}{tr}

\newcommand{\ov}{\overline}

\newcommand{\N}{\mathbb N}

\newcommand{\Z}{\mathbb Z}
\newcommand{\e}{\epsilon}

\newcommand{\mc}{\mathcal}
\newcommand{\M}{\mc{M}}
\newcommand{\R}{\mathbb{R}}
\newcommand{\mbb}{\mathbb}

\DeclareMathOperator{\Pfaff}{Pfaff}

\DeclareMathOperator{\Tr}{Tr}
\DeclareMathOperator{\detg}{det_g}
\DeclareMathOperator{\contr}{contr}

\title[Conformally Invariant Integrals]{Global Conformal Invariants of Submanifolds}
\author{Andrea Mondino}
\address[A. Mondino]{ University of Warwick, Mathematics Institute, CV4 7AL Coventry, United Kingdom.}
\email{A.Mondino@warwick.ac.uk}

\author{Huy The Nguyen}
\address[H. T. Nguyen] {Queen Mary University of London, School of Mathematical Sciences, Mile End Road, E1 4NS London, United Kingdom.}
\email{h.nguyen@qmul.ac.uk}

\begin{document} 
\maketitle
\begin{abstract} 
The goal of the present paper is to investigate  the algebraic structure of global conformal invariants of submanifolds. These are defined to be  conformally invariant integrals of geometric scalars of the tangent and normal bundle. A famous example of a global conformal invariant is the Willmore energy of a surface. In codimension one we classify such invariants, showing that under a structural hypothesis (more precisely we assume the integrand depends separately on the intrinsic and extrinsic
curvatures, and not on their derivatives)  the integrand can only consist of an intrinsic scalar conformal invariant, an extrinsic scalar conformal invariant and the Chern-Gauss-Bonnet integrand. In particular, for codimension one surfaces, we show that the Willmore energy is the  unique global conformal invariant,  up to the addition of a topological term (the Gauss curvature, giving the Euler Characteristic by the Gauss Bonnet Theorem). A similar statement holds also for codimension two surfaces, once taking  into account an additional topological term given by the Chern-Gauss-Bonnet integrand of the normal bundle. We also discuss existence and properties of natural higher dimensional (and codimensional) generalizations of the Willmore energy.
\end{abstract}

\section{Introduction} 
Let us consider an $m$-dimensional Riemannian manifold $ ( \M^m, g ^m )$ isometrically immersed in a Riemannian manifold $ ( \bar \M^n , \bar g ^ n )$. The fundamental objects describing the \emph{intrinsic} geometry of  $( \M^m, g^m )$ are the metric $  g^m$, the curvature tensor $ R _{ ijkl} $, and the Levi-Civita connection. On the other hand, the fundamental  quantities  describing  the \emph{extrinsic} geometry of $ ( \M^m, g^m) $ as  submanifold of $ ( \bar \M^n , \bar g ^ n ) $ are the second fundamental form $ h _{ij } ^ \alpha $, the normal connection $   \nabla ^ \perp$, and the normal curvature $\bar R ^ \perp_{ ij\alpha\beta} $, where Roman indices indicate tangential directions and Greek indices indicate normal ambient directions.  It is well known (see Section \ref{sec:BM} for more details) that these geometric quantities are not mutually independent but must satisfy some compatibility conditions, the so called Gauss-Codazzi-Mainardi-Ricci equations.   A natural way  to define \emph{geometric scalars} out of this list of tensors is by taking tensor products and then contracting using the metric $\bar{g}$. More precisely, we first take a  finite number of tensor products, say
\begin{equation}\nonumber
R_{i_1 j_1 k_1 l_1} \otimes \ldots \otimes R^\perp_{i_r, j_r, \alpha_r, \beta_r}  \otimes \ldots \otimes  h_{i_s j_s} \quad,  
\end{equation}
thus obtaining a tensor of rank $4+\ldots+4+\ldots+2+\ldots+2$. Then, we repeatedly pick out pairs of indices in the above expression and contract them against each other using the metric $\bar{g}^{\alpha \beta}$ (of course, in case of  contractions not including the normal curvature $R^\perp$ it is enough to contract using $g^{ij}$). This can be viewed in the more abstract perspective of Definition  \ref{def:ComplContr} by saying that we consider a \emph{geometric complete contraction}
\begin{equation}\nonumber \label{eq:ContrGen0}
C(\bar{g},R,R^\perp,h) = \contr(\bar{g}^{\alpha_1 \beta_1}\otimes \ldots \otimes R_{i_1 j_1 k_1 l_1} \otimes \ldots \otimes R^\perp_{i_r, j_r, \alpha_r, \beta_r}  \otimes \ldots \otimes  h_{i_s j_s}) \quad. 
\end{equation}
Let us stress that a complete contraction is  determined by the \emph{pattern} according to which different indices contract against each other; for example, the complete contraction $R_{ijkl}\otimes R^{ijkl}$ is different from $R^i_{ikl} \otimes R_{s}^{ksl}$. By taking linear combinations of geometric  complete contractions  (for the rigorous meaning see Definition \ref{def:LinCombCC}), we construct  \emph{geometric scalar  quantities} 
\begin{equation}\nonumber \label{eq:PgRh0}
P(\bar{g},R, R^\perp, h):= \sum_{ l \in L } a _ l C ^l (\bar{g},R, R^\perp, h) \quad.
\end{equation}
The goal of the present paper is to classify those geometric scalar quantities which, once integrated over arbitrary submanifolds $(\M^m, g^m)$ of arbitrary manifolds $(\bar{\M}^n,\bar{g}^n)$, give rise to  \emph{global conformal invariants}.  More precisely, we say that the geometric scalar quantity $P(\bar{g},R, R^\perp, h)$  is a  \emph{global conformal invariant for $m$-submanifolds in $n$-manifolds} if the following holds: for any ambient Riemannian manifold $\bar{\M}^n$, any compact orientable $m$-dimensional immersed submanifold $\M^m$ of  $\bar{\M}^n$ and any $\phi \in C^\infty(\bar{\M})$,  if one considers the conformal deformation $\hat{\bar{g}}:= e ^{2 \phi(x)} \bar{g}$ and calls $\hat{R}, \hat{R}^\perp, \hat{h}$ the tensors computed with respect to the conformal metric $\hat{\bar{g}}$,   then 
\begin{equation}\nonumber \label{eq:ConInv0}
\int_{\M^m} P(\hat{\bar{g}},\hat{R}, \hat{R}^\perp, \hat{h}) \, d \mu_{\hat{g}}= \int_{\M^m} P(\bar{g},R, R^\perp, h)\, d \mu_g \quad.
\end{equation}

Let us mention that the corresponding classification  for  \emph{intrinsic} global conformal invariants of Riemannian manifolds was a classical problem in conformal geometry motivated also by theoretical physics (the goal being to understand the so called conformal anomalies):  indeed it is the celebrated   Deser-Schwimmer conjecture  \cite{DeserSchwimmer} which has recently been solved in a series of works by Alexakis \cite{AlexI,AlexII,AlexPf1, AlexPf2, AlexIV, AlexBook}.  Inspired by the aforementioned papers, we address the problem of an analogous classification for global conformal invariants, but this time, \emph{for submanifolds}.  Of course, as explained above, if one considers global conformal invariants for \emph{submanifolds} many other \emph{extrinsic} terms appear, namely the second fundamental form, the curvature of the normal bundle, and  the normal connection; therefore the zoology of global conformal invariants is more rich and the classification more complicated. 
\medskip

Let us note that the full class of local geometric scalars is known (for a general discussion
of such invariants see \cite{Gil}), and is substantially broader than the one considered
here; they include not only contractions in the curvatures and second fundamental forms,
but also covariant derivatives of these natural tensors. Such local Riemannian invariants
arise naturally in the asymptotics of the heat kernel over Riemannian manifolds; integrals
of local invariants appear in this connection also.
\\Therefore, the problem addressed in this paper can be seen as a special case of a broader
problem of understanding the global conformal invariants of submanifolds in general.
One expects the broader problem to be even harder (in general) than the problem of understanding
the global invariants of closed Riemannian manifolds, thus the present work proposes a natural pursuit: to understand the global invariants that depend (locally) only
on the ambient curvature and the second fundamental form of the submanifold, and not
on their covariant derivatives.
\medskip

A well-known example of a global conformal invariant for two-dimensional submanifolds (called from now on surfaces) is the \emph{Willmore energy}.  For an immersed surface $ f : \M ^ 2 \rightarrow ( \bar \M^ n , \bar g^n ) $ this is defined by 
\begin{equation}\label{eq:defW}
\mc {W} ( f ) =  \int _{ \M } | H | ^ 2 d \mu _ g + \int _{ \M } \bar {K}_{ \bar \M} (Tf(\M)) d \mu _g \quad,
\end{equation}
where $H= \frac{1}{2} g^{ij} h_{ij}$ is the mean curvature vector and  $\bar {K}_{ \bar \M} (Tf(\M))$ is the sectional curvature of the ambient manifold computed on the tangent space of $f(\M)$. Clearly, in case $\bar{\M}^n=\R^n$,  \eqref{eq:defW} reduces to the familiar Euclidean Willmore energy as $ \bar K _{ \R ^ n } = 0$. It is well known that  the Willmore energy in  Euclidean space is invariant under conformal transformations of the ambient manifolds, that is M\"obius transformations where the inversion is centered off the submanifold. In fact, more generally the conformal Willmore energy is invariant under conformal deformations of the ambient background metric. This can be seen by the following decomposition, 
\begin{align*}
\|h^\circ \| ^ 2 = \| h \| ^ 2 - 2 |H| ^ 2, \quad K_{\M} = \frac{1}{2} (4  |H| ^ 2 - \| h \| ^ 2 ) + \bar K _{ \bar \M },  
\end{align*}
where $h^\circ_{ij}:=h_{ij}-H g_{ij}$ is the traceless second fundamental form,  $K_{\M}$ is the Gauss curvature of $(\M,g)$, and in the second identity we just recalled the classical Gauss equation. 
It follows that the conformal Willmore energy can be written as 
\begin{align*}
   |H| ^ 2 + \bar K_{\bar \M } = \frac 12 \| h^\circ \| ^ 2 + K_\M.
\end{align*} 
Since $ \| h^\circ \| ^ 2 d \mu_g $ is a pointwise  conformal invariant and $ \int _{\M} K_{\M} d \mu_ g= 2 \pi \chi(\M)$ is a topological (hence, a fortiori,  global conformal) invariant by the Gauss-Bonnet theorem,  clearly any linear combination of the two is a global conformal invariant. A natural question  is whether  the Willmore functional is the unique energy having such an invariance property, up to topological terms.
Let us briefly mention that the Willmore energy has recently received much attention \cite{Blaschke, CaMo, BeRi, KMS, KS, LMS, LY,  Mon2, MR2, MontielUrbano, Riv,  Sim, Will}, and in particular the Willmore conjecture in codimension one has been solved \cite{MN}. 
\medskip 

We will show that, for codimension one surfaces, any global conformal invariant of a surface must be a linear combination of the norm squared of the traceless second fundamental form and the intrinsic Gauss curvature, that is the Willmore energy is the unique global  conformal invariant  up to the Gauss-Bonnet integrand which is a topological quantity (see Theorem \ref{thm:CD1surf}). 

A similar statement holds also for codimension two surfaces (see Theorem  \ref{thm:CD2surf}), once taking into account an additional topological term given by the Chern-Gauss-Bonnet integrand of the normal bundle.

For general submanifolds of codimension one, we show that if the global conformal invariant is a polynomial in the second fundamental form only, then it must be a contraction of the \emph{traceless} second fundamental form, that is it must be the integral of a \emph{pointwise} conformal invariant (see Theorem \ref{thm:P(g,h)}). Combining this with a theorem of Alexakis \cite[Theorem 1]{AlexI}, we show that if $m$ is even and  the global conformal invariant has no mixed contractions between the intrinsic and the extrinsic curvatures then the integrand must be a linear combination of contractions of the (intrinsic) Weyl curvature, contractions of the traceless second fundamental form and the integrand of the Chern-Gauss-Bonnet formula, see Theorem \ref{thm:P=P1+P2}.

As an application of these ideas, in the last Section \ref{sec:genWill}, we introduce a higher dimensional analogue of the Willmore energy for hypersurfaces in Euclidean spaces. Such new  energies are conformally invariant and  attain the  strictly positive lower bound  only at round spheres, with rigidity, regardless of the topology of the hypersurface (see Theorem \ref{thm:3} and Theorem \ref{thm:4}).

\medskip

\begin{center}

{\bf Acknowledgments}
\\This work was written while A.M. was supported by the ETH
Fellowship and H.T.N. was visiting the  {\it Forschungsinstitut f\"ur Mathematik} at the ETH Z\"urich.  They wish to thank ETH and FIM for the hospitality and the excellent
working conditions.
\end{center}
\medskip

\section{Background Material}\label{sec:BM}
\subsection{Complete contractions: abstract definition}
Following \cite{AlexI}, in this short section we define the notion of complete contractions.
\begin{definition}[Complete Contractions]\label{def:ComplContr}
Any complete contraction 
\begin{align*}
C = \contr( ( A ^ 1 ) _{ i_ 1\dots i_s} \otimes ( A ^ t ) _{ j_1 \dots j _ q }) 
\end{align*}
will be seen as a formal expression. Each factor $ ( A ^ l ) _{ i _ 1\dots i_s } $ is  an ordered set of slots.  Given the factors $ A ^ 1 _{ i _1\dots i _ s }, A ^ l _{ j_1\dots j _ q } $ a complete contraction is then a set of pairs of slots $ ( a _1, b_1 ), \dots, (a _ w, b _ w ) $ with the following properties:
if $ k \neq l , \{a _ l, b _ l \}\cap \{ a _k ,b _k \} = \emptyset $, $ a _k \neq b _k $ and $ \bigcup_{ i = 1 } ^ { w } \{ a _i, b _i \} = \{ i _ 1 , \dots, j_q \}.$ Each pair corresponds to a particular contraction.
  
 Two complete contractions 
 \begin{align*}
\contr(( A^1 ) _{ i_1\dots i _ s } \otimes \dots \otimes ( A ^ t ) _ { j _1\dots j _w} ) 
\end{align*}
and 
\begin{align*}
\contr( (B^1)_{f _1\dots f _ q} \otimes \dots \otimes ( B ^ {t'} )_{v _ 1\dots v _ z } ) 
\end{align*}
will be identical if $ t = t ', A ^ l = B ^ l$ and if the $ \mu$-th index in $ A ^ l$ contracts against the $ \nu$-th index in $ A^r$ then the $\mu$-th index in $ B^ l $ contracts against the $\nu$-th in $ B ^ r $.

For a complete contraction, the \emph{length} refers to the number of factors.
\end{definition}

\begin{definition}[Linear combinations of complete contractions]\label{def:LinCombCC}
Linear combinations of complete contractions are defined as expressions of the form
\begin{align*}
\sum_{ l \in L } a  _ l C ^l_1, \quad \sum_{ r \in R } b _r C_2 ^ r \quad,
\end{align*}
where each $ C ^ l _i$ is a complete contraction. 
Two linear combinations are identical if $ R = L$, $ a _ l = b _l $ and $ C_1 ^ l = C_2 ^ l$. A linear combination of complete contractions is identically zero if for all $ l \in L $ we have $ a _l =0 $. For any complete contraction, we will say that a factor $ ( A)_{ r _1\dots r _{s_l} } $ has an internal contraction if two indices in 
\begin{align*}
A _{ r_1\dots r _ {s_l} } 
\end{align*} 
are contracting amongst themselves.
\end{definition}

\subsection{Riemannian and Submanifold Geometry} \label{SSS:RSG}
Consider an $n$-dimensional Riemannian manifold $(\bar{\M} ^ n ,\bar{g} ^ n )$. Given $ x _ 0 \in \bar{\M}^ n$, let   $  ( x ^1, \dots , x ^ n)$ be a local coordinate system with  associated coordinate vector fields denoted by $ X^ \alpha $, that is $ X ^ \alpha = \frac { \partial }{ \partial x ^ \alpha } $. Called $\bar{\nabla}$ the Levi-Civita connection associated to  $(\bar{\M} ^ n ,\bar{g} ^ n )$, the covariant derivative $ \bar{\nabla} _{ \frac { \partial }{\partial x ^ \alpha} }$ will be shortly denoted by  $\bar{\nabla}_\alpha$. 

The curvature tensor  $ \bar{R} _{ \alpha \beta \gamma \eta } $ of $ \bar{g} ^ {n} _{ \alpha \beta  } $ is given by the commutator of the covariant derivatives, that is 
\begin{align}\label{eq:defR}
[ \bar{\nabla} _\alpha \bar{\nabla} _\beta  - \bar{\nabla} _\beta  \bar{\nabla} _\alpha ] X_\gamma = \bar{R} _{ \alpha\beta \gamma \eta} X ^ \eta \quad,
\end{align} 
which in terms of coordinate systems may be expressed by Christoffel symbols,
\begin{align*}
\bar{R} _{ \alpha \beta \gamma} ^ \eta = \partial_{ \beta  } \Gamma _{ \alpha \gamma } ^ \eta - \partial _\gamma \Gamma ^ \eta _{ \alpha \beta  } + \sum_{ \mu  } (\Gamma_{ \alpha \gamma } ^ \mu  \Gamma _{ \mu \beta  } ^ \eta   -\Gamma _{ \alpha \beta } ^ \mu  \Gamma _{ \mu \gamma } ^ \eta ). 
\end{align*}
%The Ricci tensor is the trace of the curvature tensor $\overline{ \Ric}_{\alpha \gamma}$
%\begin{align*}
%\overline{\Ric}_{ \alpha \gamma } = \bar{R}_{ \alpha \beta \gamma\eta} g ^ {\beta \eta}.
%\end{align*} 
The two Bianchi identities are then
\begin{align*}
&\bar{R}_{ \alpha \beta \gamma\eta} + \bar{R} _{ \gamma\alpha \beta  \eta } + \bar{R} _{\beta \gamma\alpha  \eta } = 0\\
&\bar{\nabla} _ \alpha  \bar{R} _{ \beta \gamma\eta \mu } + \bar{\nabla} _\gamma \bar{R}_{\alpha \beta \eta \mu } + \bar{\nabla} _{ \beta } \bar{R} _{ \gamma\alpha \eta  \mu } = 0. 
\end{align*}
Recall that, under a conformal change of metric $\hat{\bar{g}}^n= e ^{2 \phi (x)} \bar{g}^n(x)$, the curvature transforms as follows (see for instance \cite{Eastwood}):
\begin{eqnarray}
\bar{R}^{\hat{g}^n}_{\alpha \beta \gamma\eta  }&=& e ^{2\phi (x)} \big[\bar{R}^{\bar{g}^n}_{\alpha \beta \gamma\eta  }+\bar{\nabla}_{\alpha \eta  }\phi \bar{g}^n_{\beta \gamma}+ \bar{\nabla}_{\beta \gamma} \phi \bar{g}^n_{\alpha \eta  }-\bar{\nabla}_{\alpha \gamma}\phi \bar{g}^n_{\beta \eta  }- \bar{\nabla}_{\beta \eta  } \phi \bar{g}^n_{\alpha \gamma} \nonumber \\
                              && \quad \quad  \quad+ \bar{\nabla}_\alpha  \phi \bar{\nabla}_\gamma \phi  \bar{g}^n_{\beta \eta  } + \bar{\nabla}_\beta  \phi \bar{\nabla}_\eta   \phi  \bar{g}^n_{\alpha \gamma} - \bar{\nabla}_\alpha  \phi \bar{\nabla}_\eta   \phi  \bar{g}^n_{\beta \gamma} - \bar{\nabla}_\beta  \phi \bar{\nabla}_\gamma \phi  \bar{g}^n_{\alpha \eta  } \nonumber \\
                              &&\quad \quad \quad + |\bar{\nabla} \phi|^2 \bar{g}^n_{\alpha \eta  } \bar{g}^n_{\beta \gamma} - |\bar{\nabla} \phi |^2 \bar{g}^n_{\alpha \gamma} g_{\eta  \beta }    \big] \quad .\label{eq:Rhatg} %\\
%\Ric^{\hat{g}^n}_{\alpha \beta } &=&    \Ric^{\bar{g}^n}_{\alpha \beta }+(2-n) \bar{\nabla}^2_{\alpha \beta } \phi - \Delta \phi \bar{g}^n_{\alpha \beta } + (n-2) (\bar{\nabla}_\alpha  \phi \bar{\nabla}_\beta   \phi- |\bar{\nabla} \phi|^2 \bar{g}^n_{\alpha \beta } )             \quad . \label{eq:Richatg}                
\end{eqnarray}
Now let briefly introduce some basic notions of submanifold geometry. Given an $m$-dimensional manifold $\M^m$, $2\leq m <n$, we consider $f:\M^m \hookrightarrow  \bar{\M}^n$, a smooth immersion. Recall that for every fixed $\bar{x} \in \M^m$ one can find local coordinates $(x^1, \ldots , x^n)$ of $\bar{\M}^n$ on a neighborhood  $V_{f(\bar{x})}^{\bar{\M}}$ of $f(\bar{x})$ such that $\big((x^1\circ f), \ldots, (x^m\circ f) \big)$ are local coordinates on a neighborhood  $U_{\bar{x}}^{\M}$ of $\bar{x}$ in $\M^m$ and such that 
$$f\big(U_{\bar{x}}^{\M}\big)=\Big\{(x^1, \ldots, x^n)\in V_{f(\bar{x})}^{\bar{\M}} \; : \; x^{m+1}=\ldots =x^n=0\Big \} \quad. $$
Such local coordinates on $\bar{\M}$ are said to be  \emph{adapted} to $f(\M)$. We use the convention that latin index letters vary from $1$ to $m$ and refer to geometric quantities on $\M^m$, while greek index letters vary from $1$ to $n$ (or sometimes from $m+1$ to $n$ if otherwise specified) and denote quantities in the ambient manifold $\bar{\M}^n$ (or in the orthogonal space to $f(\M)$ respectively). %In order to simplify the notation, below we will identify $\M$ with $f(\M)$ (notice that every immersion is  locally in $\M$  an embedding, so this identification makes sense   ) 
In  adapted coordinates, it is clear that $X_1, \ldots, X_m$ are a bases for the tangent space of $f(\M^m)$ and that the restriction of the ambient metric  $\bar{g}^n$ defines an induced  metric on $\M^m$, given locally by 
$$g^m_{i j}:= \bar{g}^n(X_i, X_j)\quad .$$
Using standard notation, $(g^m)^{ij}$ denotes the inverse of the matrix $(g_m)_{ij}$, that is $(g^m)^{ik} (g^m)_{kj}=\delta_{ij}$.
For every $\bar{x} \in \M$, the ambient metric $\bar{g}^n$ induces the  orthogonal splitting
$$T_{f(\bar{x})} \bar{\M}= T_{f(\bar{x})} f(\M) \oplus  [T_{f(\bar{x})} f(\M)]^\perp \quad,  $$
where, of course,  $ [T_{f(\bar{x})} f(\M)]^\perp$ is the orthogonal complement of the $m$-dimensional subspace $T_{f(\bar{x})} f(\M) \subset T_{f(\bar{x})} \bar{\M}$.
We call $\pi_T: T_{f(\bar{x})} \bar{\M} \to T_{f(\bar{x})} f(\M)$ and $\pi_N=Id-\pi_T: T_{f(\bar{x})} \bar{\M} \to [T_{f(\bar{x})} f(\M)]^\perp$ the tangential and the normal projections respectively, one can define the \emph{tangential} and the \emph{normal connections} (which correspond to the Levi-Civita connections on $(\M, g)$ and on the normal bundle respectively)  by
\begin{eqnarray}\label{eq:tgNConn}
\nabla_{X_i} X_j&:=&\pi_T(\bar{\nabla}_{X_i} X_j), \; i,j=1,\ldots,m, \label{eq:tgConn} \\
\nabla^\perp _{X_i} X_{\alpha} &:=&\pi_N(\bar{\nabla}_{X_i} X_\alpha), \; i=1,\ldots,m, \alpha=m+1, \ldots, n.  \label{eq:NConn}
\end{eqnarray}
Associated to the tangential and normal connections we have the tangential and normal Riemann curvature tensors (which correspond to the curvature of $(\M, g)$ and of the normal bundle respectively) defined analogously to \eqref{eq:defR}:
\begin{eqnarray}
[ \nabla _i \nabla _j  - \nabla _j  \nabla _i ] X_k &= &R _{ i j k l} X ^ l \;, \quad i,j,k,l=1,\ldots, m \label{eq:defRt} \\
\left[{\nabla^\perp_i \nabla^\perp_j -  \nabla^\perp_j \nabla^\perp_i} \right] X_{\alpha} &= &R^\perp _{ i j \alpha \beta} X ^ \beta \;, \quad i,j=1,\ldots, m, \; \alpha, \beta=m+1, \ldots, n. \label{eq:defRn} 
\end{eqnarray}
The transformation of $R_{ijkl}$ and $R^\perp_{ij\alpha \beta}$ under a conformal change of metric  is analogous to  \eqref{eq:Rhatg}, just replacing $\bar{\nabla}$ with $\nabla$ or with $\nabla^{\perp}$ respectively.
The second fundamental form $h$ of $f$ is defined by
\begin{equation}\label{eq:defh}
h(X_i, X_j):= \pi_{N} (\bar{\nabla}_{ X_i} X_j)= \bar{\nabla}_{X_i} X_j- \nabla_{X_i} X_j  \quad .
\end{equation}
It can be decomposed orthogonally into its trace part, the \emph{mean curvature}
\begin{equation}\label{eq:defH}
H:=\frac{1}{m} (g^m)^{ij} h_{ij} \quad,
\end{equation}
and its trace free part, the \emph{traceless second fundamental form}
\begin{equation}\label{eq:hH}
h^\circ_{ij}:=h_{ij}-H g^m_{ij} \quad,
\end{equation}
indeed it is clear from the definitions that
\begin{equation}\label{eq:Splith}
h_{ij}=  h_{ij}^\circ + H g^m_{ij} \quad \text{ and } \quad \langle h^\circ, Hg^m \rangle = (g^m)^{ik} (g^m)^{jl} \; h^\circ_{ij} \; H g_{kl}=H \;  \Tr_{g^m} (h^\circ) =0 \quad.
\end{equation}
Under a conformal change of the ambient metric  $\hat{\bar{g}}^n= e ^{2 \phi (x)} \bar{g}^n(x)$, the above quantities change as follow:
\begin{equation}\label{eq:CChH}
h^{\hat{g}}_{ij}= e^{\phi} \big[h_{ij}-g^{m}_{ij}\; \pi_{N} (\bar{\nabla} \phi) \big], \quad H^{\hat{g}}=e^{- \phi} \left[ H -  \pi_{N} (\bar{\nabla} \phi) \right] \quad\text{and}\quad (h^{\hat{g}})^{\circ}_{ij}= e^{\phi}\; h^{\circ}_{ij} \quad.  
\end{equation}
Observe that, in particular, the endomorphism of $T f(\M)$ associated to $h^{\circ}$ is invariant under conformal deformation of the ambient metric, that is 
\begin{equation}
 [(h^{\hat{g}})^{\circ}]^i_j= [h^{\circ}]^i_j \quad.
 \end{equation}
 Finally let us recall the fundamental equations of Gauss and Ricci which link the ambient curvature $\bar{R}$ computed on $T f (\M)$ (respectively on $T f (\M)^{\perp}$) with the second fundamental form and the intrinsic curvature $R$ (respectively with the second fundamental form and the normal curvature $R^\perp$):
 \begin{eqnarray} 
 \bar{R}_{ijkl}&=&R_{ijkl}+(h_{il})_{\alpha} \; (h_{jk})^{\alpha}- (h_{ik})_{\alpha} \;  (h_{jl})^{\alpha} \quad,  \label{eq:Gauss} \\
 \bar{R}_{ij \alpha \beta}&=&R_{ij\alpha \beta}^\perp  + (h_{ik})_\alpha \; (h_j^k)_{\beta}-(h_{ik})_{\beta} \; (h_j^k)_{\alpha} \quad, \label{eq:Ricci}
 \end{eqnarray}
where, of course, $(h_{il})^{\alpha}$ denotes the $\alpha$-component of the vector $h_{il}\in T f (\M)^{\perp}\subset T \bar{\M}$ and where we adopted Einstein's convention on summation of repeated indices.

\subsection{Geometric complete contractions and global conformal Invariants of Submanifolds}
Given an immersed submanifold $f:\M^m\hookrightarrow \bar{\M}^n$, of course the above defined Riemannian curvature $R$, the curvature of the normal bundle $R^\perp$, and the extrinsic curvatures $h, h^\circ, H$ are geometric objects, that is they are invariant under change of coordinates and under isometries of the ambient manifold.  So they give a list of \emph{geometric tensors}. A natural way  to define \emph{geometric scalars} out of this list of tensors is by taking tensor products and then contracting using the metric $\bar{g}$. More precisely, we first take a  finite number of tensor products, say,
\begin{equation}
R_{i_1 j_1 k_1 l_1} \otimes \ldots \otimes R^\perp_{i_r, j_r, \alpha_r, \beta_r}  \otimes \ldots \otimes  h^{\circ}_{i_s j_s} \otimes \ldots \otimes H g_{i_t j_t} \quad,  
\end{equation}
thus obtaining a tensor of rank $4+\ldots+4+\ldots+2+\ldots+2$. Then, we repeatedly pick out pairs of indices in the above expression and contract them against each other using the metric $\bar{g}^{\alpha \beta}$ (of course, in case of  contractions not including the normal curvature $R^\perp$ it is enough to contract using $g^{ij}$). This can be viewed in the more abstract perspective of Definition  \ref{def:ComplContr} by saying that we consider a \emph{geometric complete contraction}
\begin{equation}\label{eq:ContrGen}
C(\bar{g},R,R^\perp,h) = \contr(\bar{g}^{\alpha_1 \beta_1}\otimes \ldots \otimes R_{i_1 j_1 k_1 l_1} \otimes \ldots \otimes R^\perp_{i_r, j_r, \alpha_r, \beta_r}  \otimes \ldots \otimes  h^{\circ}_{i_s j_s} \otimes \ldots \otimes H g_{i_t j_t} ) \quad. 
\end{equation}
Let us stress that a complete contraction is  determined by the \emph{pattern} according to which different indices contract against each other; for example, the complete contraction $R_{ijkl}\otimes R^{ijkl}$ is different from $R^i_{ikl} \otimes R_{s}^{ksl}$. By taking linear combinations of geometric  complete contractions  (for the rigorous meaning see Definition \ref{def:LinCombCC}), we construct  \emph{geometric scalar  quantities} 
\begin{equation}\label{eq:PgRh}
P(\bar{g},R, R^\perp, h):= \sum_{ l \in L } a _ l C ^l (\bar{g},R, R^\perp, h) \quad.
\end{equation}

\begin{remark}
As already mentioned in the introduction, let us note that the full class
of geometric scalar quantities is known (for a general discussion of such invariants see
\cite{Gil}), and is larger than the one considered here: they include not only contractions in
the curvatures and second fundamental forms, but also covariant derivatives of these
natural tensors.
\end{remark}

\begin{remark}
Notice that thanks to the Gauss \eqref{eq:Gauss} and Ricci \eqref{eq:Ricci} equations, one can express the ambient curvature $\bar{R}$ restricted on the tangent space  of $\M$ or restricted to the normal bundle (that is $\bar{R}_{ijkl}$ and $\bar{R}_{ij\alpha \beta}$) as a quadratic combination of  $R, R^\perp$ and $h$. This is the reason why we can assume it is not present  in the complete contractions  \eqref{eq:ContrGen} without losing generality. Analogously, thanks to  \eqref{eq:hH}, we can assume that $h$ is not present in the complete contractions but just $h^\circ$ and  $H g$.
\end{remark}

\begin{definition}[Weight of a geometric scalar quantity]
Let $P(\bar{g},R, R^\perp, h)$ be a  geometric scalar quantity  as in \eqref{eq:PgRh} and consider the homothetic rescaling $\bar{g} \mapsto t^2 \bar{g}$ of the ambient metric $\bar{g}$, for $t \in \R_{+}$. By denoting $R_{t^2\bar{g}}, R^\perp_{t^2\bar{g}}, h_{t^2\bar{g}}$ the tensors computed with respect to the rescaled metric $t^2 \bar{g}$,    if
$$P(t^2\bar{g},R_{t^2\bar{g}}, R^\perp_{t^2\bar{g}}, h_{t^2\bar{g}})=  t^K  P(\bar{g},R_{\bar{g}}, R^\perp_{\bar{g}}, h_{\bar{g}}), \quad \text{for some } K \in \Z \quad,$$
we then say that  $P(\bar{g},R, R^\perp, h)$ is  a  geometric scalar quantity of \emph{weight $K$}.
\end{definition}
Recall that, under a general conformal deformation $\hat{\bar{g}}= e ^{2 \phi(x)} \bar{g}$ of the ambient metric $\bar{g}$ on $\bar{\M}^n$, the volume form of the immersed $m$-dimensional submanifold $f(\M)$ rescales by the formula  $d\mu_{\hat{g}}=e ^{m \phi(x)}  d\mu_{g}$, where of course  $\hat{g}= e ^{2 \phi(x)} g$ is the induced conformal deformation on $f(\M)$. In particular, for every constant $t\in \R_+$, we have $d\mu_{t^2 g}= t^m d \mu_{g}$. Thus, for any scalar geometric quantity $P(\bar{g},R, R^\perp, h)$ of weight  $-m$, the integral $\int_{\M^m} P(\bar{g},R, R^\perp, h) \, d\mu_g$ is scale invariant for all compact orientable $m$-dimensional immersed submanifolds in any $n$-dimensional ambient Riemannian manifold.  We are actually interested in those scalar geometric quantity $P(\bar{g},R, R^\perp, h)$ of weight  $-m$ which give rise to integrals which are invariant not only under constant rescalings, but under general conformal rescalings.  Let us give a precise definition.

\begin{definition}[Global conformal invariants of submanifolds]\label{def:GCI}
Let $P(\bar{g},R, R^\perp, h)$ be a  geometric scalar quantity  as in \eqref{eq:PgRh} of weight $-m$ and consider the conformal rescaling $\hat{\bar{g}}:= e ^{2 \phi(x)} \bar{g}$ of the ambient metric $\bar{g}$, for $\phi \in C^\infty(\bar{\M})$. By denoting $\hat{R}, \hat{R}^\perp, \hat{h}$ the tensors computed with respect to the conformal metric $\hat{\bar{g}}$,   if 
\begin{equation}\label{eq:ConInv}
\int_{\M^m} P(\hat{\bar{g}},\hat{R}, \hat{R}^\perp, \hat{h}) \, d \mu_{\hat{g}}= \int_{\M^m} P(\bar{g},R, R^\perp, h)\, d \mu_g 
\end{equation}
for any ambient Riemannian manifold $\bar{\M}^n$, any compact orientable $m$-dimensional immersed submanifold $f(\M^m)\subset \bar{\M}^n$ and any $\phi \in C^\infty(\bar{\M})$, we then say that $\int_{\M^m} P(\bar{g},R, R^\perp, h)\, d \mu_g$ is a \emph{global conformal invariant for $m$-submanifolds in $n$-manifolds}.
\end{definition}

In this paper we address the question of classifying (at least in some cases) such global conformal invariants of submanifolds.

\subsection{The Operator $ I _{\bar{g},R, R^\perp, h} (\phi) $ and its Polarisations}  \label{SS:Ig}
Inspired by the work of Alexakis  \cite{AlexI,AlexII,AlexPf1, AlexPf2, AlexIV, AlexBook} on the classification of global conformal invariants of Riemannian manifolds, we introduce the useful operator $I_{\bar{g},R, R^\perp, h}(\phi)$, which ``measures how far the scalar geometric invariant    $P(\bar{g},R, R^\perp, h)$ is from being a global conformal invariant''. 
\begin{definition}\label{def:Iphi}
Let  $P(\bar{g},R, R^\perp, h)$ be a linear combination 
$$P(\bar{g},R, R^\perp, h):= \sum_{ l \in L } a _ l C ^l (\bar{g},R, R^\perp, h) \quad,$$
where each $C ^l (\bar{g},R, R^\perp, h)$ is  in the form \eqref{eq:ContrGen} and has weight $-m$, and assume that $P(\bar{g},R, R^\perp, h)$ gives rise to a global conformal invariant for $m$-submanifolds, that is it satisfies \eqref{eq:ConInv}.  We define the differential operator $I _{\bar{g},R, R^\perp, h} (\phi)$, which depends both on the geometric tensors $\bar{g},R, R^\perp, h$ and on the auxiliary function $\phi\in C^{\infty}(\bar{\M}^n)$ as
\begin{equation}\label{eq.defI}
I _{\bar{g},R, R^\perp, h} (\phi)(x):= e^{m \phi(x)} P(\hat{\bar{g}},\hat{R}, \hat{R}^\perp, \hat{h})(x)-P(\bar{g},R, R^\perp, h)(x)\quad,
\end{equation}
where we use the notation of Definition \ref{def:GCI}.
\end{definition}
Notice that, thanks to \eqref{eq:ConInv}, it holds
\begin{equation}\label{eq:IntIInv}
\int_{\M^m}  I _{\bar{g},R, R^\perp, h} (\phi) \, d \mu_g =0 \quad,
\end{equation}
for every Riemannian $n$-manifold $\bar{\M}^n$, every compact orientable $m$-submanifold $f(\M^m)\subset \bar{\M}^n$, and every function $\phi \in C^\infty(\bar{\M})$.
\\
By using the transformation laws for $R, R^\perp, h$ under conformal rescalings recalled in Section \ref{SSS:RSG},  it is clear that  $I _{\bar{g},R, R^\perp, h} (\phi)$ is a differential operator acting on the function $\phi$. In particular we can polarize, that is we can pick any $A>0$ functions $\psi_1(\cdot), \ldots, \psi_{A}(\cdot)$, and choose
$$\phi(x):=\sum_{l=1}^A \psi_l(x)\quad.$$
Thus, we have a differential operator  $I _{\bar{g},R, R^\perp, h} (\psi_1, \ldots, \psi_A)(\cdot)$ so that, by \eqref{eq:IntIInv}, it holds
$$\int_{\M^m}  I _{\bar{g},R, R^\perp, h} (\psi_1, \ldots, \psi_A) \, d \mu_g =0 \quad, $$
for every Riemannian $n$-manifold $\bar{\M}^n$, every compact orientable $m$-submanifold $f(\M^m)\subset \bar{\M}^n$, and any functions $\psi_1, \ldots, \psi_A  \in C^\infty(\bar{\M})$.
\\
Now, for any given functions $\psi_1(\cdot),\ldots, \psi_{A}(\cdot)$, we can consider the rescalings
$$\lambda_1 \psi_1(\cdot), \ldots, \lambda_{A} \psi_{A}(\cdot), $$
and, as above, we have the equation 
\begin{equation}\label{eq:IntLambda}
\int_{\M^m}  I _{\bar{g},R, R^\perp, h} (\lambda_1 \psi_1, \ldots, \lambda_A \psi_A) \, d \mu_g =0 \quad.
\end{equation}
The trick here is to see $\int_{\M^m}  I _{\bar{g},R, R^\perp, h} (\lambda_1 \psi_1, \ldots, \lambda_A \psi_A) \, d \mu_g$ as a polynomial $\Pi(\lambda_1, \ldots, \lambda_A)$ in the independent variables $\lambda_1, \ldots, \lambda_A$. But then, equation \eqref{eq:IntLambda} implies that such polynomial  $\Pi(\lambda_1, \ldots, \lambda_A)$ is identically zero. Hence, each coefficient of each monomial in the variables $\lambda_1, \ldots, \lambda_A$ must vanish. 

We will see later in the proofs of the results how to exploit this crucial trick.

\section{Global Conformal Invariants of Surfaces} 

It is well known that in the Euclidean space $ \R ^ n $, the Willmore energy $ \mc {W} ( f):=\int |H|^2 d \mu_g$ of a surface is invariant under conformal maps, that is under  M\"obius transformations with inversion centred off the surface. It can be shown that this is a consequence of the fact that the conformal Willmore energy of a surface immersed in a general Riemannian manifold $ f :\M^2 \hookrightarrow \bar{\M} ^ n $ 
\begin{align*}
\mc W_{conf}(f) =   \int_{ \M} | H| ^ 2 d \mu _g + \int _{ \M} \bar K d \mu_g 
\end{align*}  
where $ \ov K $ is the sectional curvature of the ambient space restricted to the surface, is invariant under conformal deformations of the ambient metric. Recall that by the Gauss equation one can write the intrinsic Gauss curvature $K$ of $(\M^2,g)$ as  $K =  ( |H| ^ 2 - \frac{1}{2} |h^\circ|^2 ) + \ov K $, where $ | h^\circ | ^ 2 = g^{ik}g^{jl} h^\circ_{ij} h^\circ_{kl}$ is the squared norm of the traceless second fundamental form, we can rewrite $\mc W_{conf}$ as 
$$\mc W_{conf}(f)=\frac{1}{2} \int _{ \M} | h^\circ | ^ 2 d \mu _g + \int _{ \M } K d \mu _g \quad . $$
Notice that both  $\int _{ \M} | h^\circ | ^ 2 d \mu _g$ and  $\int _{ \M } K d \mu _g$ are natural conformal invariants: the first integral is conformally invariant as $ |h^\circ  |^ 2 d \mu _g $ is a pointwise conformally invariant thanks to formula \eqref{eq:CChH}, and the second integral is conformally invariant by the Gauss-Bonnet theorem (more generally it is a \emph{topological} invariant). It trivially follows that any linear combination of the two integrands gives rise to a global conformal invariant.  Our next result is that actually in codimension one there are no other global conformal invariants and in codimension two the situation is analogous once also the normal curvature is taken into account.

\subsection{Global Conformal Invariants of Codimension One Surfaces}
As announced before,  as first result we show that, for codimension one surfaces, any global conformal invariant must be a linear combination of the squared norm of the traceless second fundamental form and the intrinsic Gauss curvature.

\begin{theorem}\label{thm:CD1surf}
Let $ P ( \bar{g} ,R, R^\perp, h  )=\sum_{l\in L} a_l C^l(\bar{g},R,R^\perp,h)  $ be a geometric scalar quantity for two-dimensional submanifolds  of codimension one made by linear combinations of complete contractions in the general form
\begin{equation}\label{eq:ContrGen1}
C^l(\bar{g},R,R^\perp,h) = \contr(\bar{g}^{\alpha_1 \beta_1}\otimes \ldots \otimes R_{i_1 j_1 k_1 l_1} \otimes \ldots \otimes R^\perp_{i_r, j_r, \alpha_r, \beta_r}  \otimes \ldots \otimes  h^{\circ}_{i_s j_s} \otimes \ldots \otimes H g_{i_t j_t} ) \quad, 
\end{equation}
and  assume that $ \int _{\M}P ( \bar{g} ,R, R^\perp, h  ) d \mu_{g}$ is a global conformal invariant, in the sense of Definition \ref{def:GCI}. 

Then there exist $a,b \in \R$ such that $ P$ has the following decomposition
\begin{align*}
P ( \bar{g} ,R, R^\perp, h  )= a K  + b | h^\circ |^ 2 \quad.
\end{align*}
\end{theorem}

\begin{proof}
First of all  notice that since by assumption here we are working in codimension one, the normal curvature $R^\perp$ vanishes identically so it  can be suppressed  without losing generality.
Moreover, by definition, $P(\bar{g}, R, h)$ is a linear combination of complete contractions  $C^ l (\bar{g},R,h)$ each  of weight $-2$. Observing that any contraction of $g^{-1}\otimes g^{-1} \otimes R$ is already of weight $-2$ and that any  contraction of $g^{-1} \otimes h$ is of weight $-1$, the only possibility for $P ( \bar{g} ,R ,h  )$ to be of weight $-2$ is that it decomposes as
\begin{eqnarray}
P(\bar{g},R,h)&=& \quad a \contr(g^{i_1j_1}\otimes g^{i_2 j_2} \otimes R_{i_3 j_3 i_4 j_4}) \nonumber \\
&&+ b \contr(g^{i_1j_1}\otimes g^{i_2 j_2}\otimes H g_{i_{3} j_{3}}\otimes \ldots \otimes H g_{i_{r+2} j_{r+2}}\otimes h^\circ_{i_{r+3} j_{r+3}}\otimes \ldots \otimes  h^\circ_{i_{r+4} j_{r+4}}  ),  \label{eq:Psplits}
\end{eqnarray} 
where in the above formula we intend that $r=0,1,2$ is the number of $H$ factors, and $2-r$ is the number of $h^\circ$ factors.
Clearly, since by assumption $\M$ is a 2-d manifold, the term  $\contr(g^{i_1j_1}\otimes g^{i_2 j_2} R_{i_3 j_3 i_4 j_4})$ is a (possibly null) multiple of the Gauss curvature.  To get the thesis it is therefore sufficient to prove that $r=0$, that is the second summand  in  \eqref{eq:Psplits} is completely expressed in terms of the traceless second fundamental form. Indeed any complete contraction of $g^{-1}\otimes g^{-1} \otimes h^\circ$ is a (possibly null) multiple of $|h^\circ|^2$.

To this aim observe that, since by the Gauss-Bonnet Theorem $\int_{\M} K d \mu_g$ is a global conformal invariant, called 
\begin{eqnarray}
P_1(g, h)&:=& P(\bar{g},R,h)- a K \nonumber \\
                        &=&  b \contr(g^{i_1j_1}\otimes g^{i_2 j_2}\otimes H g_{i_{3} j_{3}}\otimes \ldots \otimes H g_{i_{r+2} j_{r+2}}\otimes h^\circ_{i_{r+3} j_{r+3}}\otimes \ldots \otimes  h^\circ_{i_{r+4} j_{r+4}}  ), \label{eq:defP1gh}
\end{eqnarray}
we have that $\int_{\M} P_1(g,h) d \mu_g$ is a global conformal invariant for compact surfaces immersed in 3-manifolds, as difference of such objects.
\\Consider then an arbitrary compact surface $f(\M^2)$  immersed into an arbitrary  Riemannian $3$-manifold $(\bar{\M}^3, \bar{g}^3)$, and an arbitrary conformal rescaling $\hat{\bar{g}}^3=e^{2 \phi(x)} \bar{g}^3$ of the ambient metric by a smooth function $\phi\in C^\infty(\bar{\M})$. All the hatted geometric tensors $\hat{g}, \hat{h}, \hat{h}^\circ, \hat H$ denote the corresponding tensors computed with respect to the rescaled metric $\hat{\bar{g}}$. In Section \ref{SS:Ig} we defined the operator
$$I^1_{(g,h)}(\phi):= e^{2 \phi} P_1(\hat{g}, \hat{h})- P_1(g,h)\quad ,$$
and we observed that  the conformal invariance of the integrated quantities implies (see \eqref{eq:IntIInv}) 
\begin{equation}\label{eq:IntI=00}
\int_{\M} I^1_{(g,h)}(\phi) \, d \mu_{g}=0 \quad.
\end{equation}
From the formulas \eqref{eq:CChH} of the change of $h^\circ$ and $H$ under a conformal deformation of metric, it is clear that $I^1_{(g,h)}(\phi)$ does not depend directly on $\phi$ but just on $\pi_N(\bar{\nabla} \phi)= (\partial_N \phi) N $, the normal derivative of $\phi$. More precisely  $I^1_{(g,h)}(\phi)$ is a  polynomial in $\partial_N \phi$ exactly of  the same degree $0\leq r\leq 2$ as $P_1(g, h^\circ, H)$ seen as a polynomial in $H$.  

By considering $t\phi$ for $t \in \R$, we therefore get that 
$$I^1_{g, h}(t \phi)= \sum_{k=1}^r a_k C^k(g, h^\circ, H) \; t^k \left(\partial_N \phi \right)^k \quad.$$
Recalling now \eqref{eq:IntI=00}, we obtain that  $ \int_{\M} I^1_{(g,h)}(t\phi) \, d \mu_{g} $ vanishes identically as a polynomial in $t$, so
\begin{equation}\label{eq:IntIk0}
0=\frac{d^k}{dt^k}\big|_{{t=0}}  \int_{\M} I^1_{(g,h)}(t\phi) \, d \mu_{g}  = k!  \int_{\M}  a_k C^k(g, h^\circ, H) \left( \partial_N \phi \right)^k   \, d \mu_{g},\quad \forall k=0,\ldots,r.
\end{equation}
Pick an arbitrary point $x \in \M$; by choosing local coordinates in $\bar{\M}^3$ adapted to $f(\M)$ at $f(x)$, it is easy to see that for any given function $\psi \in C^{\infty}_c(\M)$ supported in such coordinate neighborhood  of $x$, there exists $\phi \in C^\infty(\bar{\M})$ such that
 $$\psi = \frac{\partial \phi}{\partial x^3}  \circ f=  \partial_N \phi   \circ f \quad. $$
 By plugging such arbitrary $C^\infty_c(\M)$ function   $\psi$ in place of  $\partial_N \phi$ in \eqref{eq:IntIk0}, we obtain that not only the integrals  but the integrands themselves must vanish, that is $a_k C^k(g, h^\circ, H) \equiv 0$ on $\M$. It follows that $I^1_{g,h}(\phi)\equiv 0$ or, in other words, the degree of  $I^1_{g, h}(t \phi)$ as a polynomial in $t$ is $0$. By the above discussion we have then that $r=0$, which was our thesis.
 \end{proof}

\subsection{Global Conformal Invariants of Codimension Two Surfaces}

Let $ (\bar{\M} ^ 4, \bar{g} _{ \alpha \beta } ) $ be a four dimensional Riemannian manifold and $ f : \M \hookrightarrow \bar{\M}$ an immersion of an oriented compact surface $\M^2$. 
%Let $ T ^ \perp \Sigma$ denote the normal bundle of $ f $ and then we have the decomposition
%\begin{align*}
%f ^ * T \M= T \Sigma \oplus T ^ \perp \Sigma.
%\end{align*} 
%Associated to this decomposition, is a splitting of the Levi-Civit\`a connection $ \bar \nabla $ of $ T \M $ into tangential and normal connections,
%\begin{align*}
%\bar \nabla = \nabla + \nabla ^ \perp.
%\end{align*}
In local coordinates, if  $ \{ e _1, e _ 2, e _ 3 , e _ 4 \} $ is an adapted orthonormal frame, so that $ \{ e _1, e _2 \} $ is a frame on $ T f(\M)$ and $ \{ e _3, e _ 4\} $ is a frame on $ [T  f (\M)]^\perp$  we  define 
\begin{align*}
K ^ \perp = R ^ \perp( e _1, e _2, e _3, e _ 4 )\quad,
\end{align*}
where $ R ^ \perp$ is the curvature tensor of the normal connection defined in \eqref{eq:defRn}. Note that  $K^\perp$ is well defined up to a sign depending on orientation,  indeed by the symmetries of the normal curvature tensor, for a codimension two surface  $R^\perp$ has only one non-zero component, namely $\pm K ^ \perp$.  
We also denote by 
$$ \bar K^ \perp = \bar R ( e _1, e _2 , e _ 3 ,e  _4 )$$
the ambient curvature evaluated on the normal bundle.

We note here  that the normal curvature is not a complete contraction of the form described above. The anti-symmetry of the curvature means that the normal curvature does not appear as a complete contraction. In fact, the normal curvature is an odd invariant (as opposed to an even invariant like the Gauss curvature). The normal curvature however is an even invariant if we include the volume form. To that end, we consider odd invariants as defined below.
\begin{definition}[(Odd) Complete Contractions \cite{Bailey1994}]\label{def:ComplContr}
Denote with $\e$ the volume form of $\R^{4}$.  Any complete contraction 
\begin{align*}
C = \contr( \e\otimes( A ^ 1 ) _{ i_ 1\dots i_s} \otimes ( A ^ t ) _{ j_1 \dots j _ q }) 
\end{align*}
will be seen as a formal expression.  Each factor $ ( A ^ l ) _{ i _ 1\dots i_s } $ is  an ordered set of slots.  Given the factors $ A ^ 1 _{ i _1\dots i _ s }, A ^ l _{ j_1\dots j _ q }$, a complete contraction is then a set of pairs of slots $ ( a _1, b_1 ), \dots, (a _ w, b _ w ) $ with the following properties:
if $ k \neq l , \{a _ l, b _ l \}\cap \{ a _k ,b _k \} = \emptyset $, $ a _k \neq b _k $ and $ \bigcup_{ i = 1 } ^ { w } \{ a _i, b _i \} = \{ i _ 1 , \dots, j_q \}.$ Each pair corresponds to a particular contraction.
  
 Two complete contractions 
 \begin{align*}
\contr(( A^1 ) _{ i_1\dots i _ s } \otimes \dots \otimes ( A ^ t ) _ { j _1\dots j _w} ) 
\end{align*}
and 
\begin{align*}
\contr( (B^1)_{f _1\dots f _ q} \otimes \dots \otimes ( B ^ {t'} )_{v _ 1\dots v _ z } ) 
\end{align*}
will be identical if $ t = t ', A ^ l = B ^ l$ and if the $ \mu$-th index in $ A ^ l$ contracts against the $ \nu$-th index in $ A^r$ then the $\mu$-th index in $ B^ l $ contracts against the $\nu$-th in $ B ^ r $.

For a complete contraction, the \emph{length} refers to the number of factors.
\end{definition}
   
The definitions for an odd complete contraction are exactly the same as for contractions as defined above with the obvious changes.  

% The normal curvature relates to the second fundamental form, traceless second fundamental form and the ambient curvature tensor as follows
%\begin{align*}
%K^\perp(e_i, e _j ,e _\alpha ,e _\beta) &= \sum_{p}(h _{ i p \alpha} h _{ j p \beta } - h _{ j p \alpha}h _{ ip\beta } )+ \bar R ( e _i, e _ j, e _\alpha, e _\beta )\\
%&=\sum_{p}(\circo h _{ i p \alpha} \circo h _{ j p \beta } - \circo h _{ j p \alpha}\circo h _{ ip\beta } )+ \bar R ( e _i, e _ j, e _\alpha, e _\beta )
%\end{align*}

For a codimension two surface, we note that we have two topological invariants: the integral of the Gauss curvature $K$ giving the Euler Characteristic of $\M$ via  the Gauss-Bonnet Theorem, and the integral of the normal curvature $K^\perp$ which gives the Euler characteristic of the normal bundle,
\begin{equation}\label{eq:TopInv}
\int_{\M} K d \mu _g = 2 \pi \chi(\M) , \quad  \int _{\M} K ^ \perp d \mu _g = 2 \pi \chi^ \perp (f(\M)).
\end{equation}
As already observed,  $|h^ \circ|^2 d \mu_g$ is a \emph{pointwise} conformal invariant. Hence any linear combination of these three integrands is a global conformal invariant and  in the next theorem we show that there are no others.

% \begin{align*}
%P (  g^ n , h ^ n , R ^ \perp) = \sum_{l \in L } a _l C^l(g^n, h^n, R ^ \perp)   
%\end{align*}
%where $ C ^l$ is a contraction of the form
%\begin{align*}
%\contr( R _{ i _1j_1k_1l_1}\otimes \dots R _{ i_p j_p k_p l_p}  \otimes h _{ i'_1j'_1} \otimes \dots \otimes h _{i'_qj'_q} \otimes R ^ \perp_{\bar i _1 \bar j_1 \alpha_1\beta_1}\otimes \dots R ^ \perp_{\bar i_r \bar j_r \alpha_r \beta_r} ) .  
%\end{align*}
%The contraction $P$ is of weight $-n$ if $ 2 p + q + 2 r = n$.

\begin{theorem}\label{thm:CD2surf}
Let $ P ( \bar{g} ,R, R^\perp, h  )=\sum_{l\in L} a_l C^l(\bar{g},R,R^\perp,h)  $ be a geometric scalar quantity for two-dimensional submanifolds  of codimension two made by linear combinations of (possibly odd) complete contractions in the general form
\begin{equation}\label{eq:ContrGen1}
C^l(\bar{g},R,R^\perp,h) = \contr(\e\otimes\bar{g}^{\alpha_1 \beta_1}\otimes \ldots \otimes R_{i_1 j_1 k_1 l_1} \otimes \ldots \otimes R^\perp_{i_r, j_r, \alpha_r, \beta_r}  \otimes \ldots \otimes  h^{\circ}_{i_s j_s} \otimes \ldots \otimes H g_{i_t j_t} ) \quad, 
\end{equation}
where $ \e$ is the volume form of $ \R^4$ and assume that $ \int _{\M}P ( \bar{g} ,R, R^\perp, h  ) d \mu_{g}$ is a global conformal invariant, in the sense of Definition \ref{def:GCI}. 

Then there exist $a,b,c \in \R$ such that $ P$ has the following decomposition
\begin{align*}
P ( \bar{g} ,R, R^\perp, h  )= a K +b K^\perp  + c | h^\circ |^ 2 \quad.
\end{align*}
\end{theorem}

\begin{proof}
Exactly as in the proof of Theorem \ref{thm:CD1surf}, if   $P ( \bar{g} ,R, R^\perp ,h  ) $ gives rise to a global conformal invariant it must necessarily be a linear combination of complete contractions  $C^ l (\bar{g},R, R^\perp, h)$ each  of weight $-2$. Observing that any contraction of $g^{-1}\otimes g^{-1} \otimes R$ and of $\e\otimes \bar{g}^{-1}\otimes \bar{g}^{-1} \otimes R^\perp$  is already of weight $-2$,  and that any  contraction of $g^{-1} \otimes h$ is of weight $-1$, the only possibility for $P ( \bar{g} ,R, R^\perp, h )$ to be of weight $-2$ is that it decomposes as
\begin{eqnarray}
P(\bar{g},R, R^\perp, h)&=& \quad a \contr(g^{i_1j_1}\otimes g^{i_2 j_2} \otimes R_{i_3 j_3 i_4 j_4}) + b \contr(\bar{g}^{\alpha_1 \beta_1}\otimes g^{\alpha_2 \beta_2} \otimes R^\perp_{\alpha_3 \beta_3 \alpha_4 \beta_4})   \nonumber \\
&&+ c \contr(\e\otimes g^{i_1j_1}\otimes g^{i_2 j_2}\otimes H g_{i_{3} j_{3}}\otimes \ldots \otimes H g_{i_{r+2} j_{r+2}} \nonumber \\
&& \quad \quad \quad  \quad \quad  \quad \quad \quad \quad \quad \; \otimes h^\circ_{i_{r+3} j_{r+3}}\otimes \ldots \otimes  h^\circ_{i_{r+4} j_{r+4}}  ),  \nonumber 
\end{eqnarray} 
where in the above formula we intend that $r=0,1,2$ is the number of $H$ factors, and $2-r$ is the number of $h^\circ$ factors.
Clearly, since by assumption $\M$ is a 2-d manifold, the term  $\contr(g^{i_1j_1}\otimes g^{i_2 j_2} R_{i_3 j_3 i_4 j_4})$ is a (possibly null) multiple of the Gauss curvature.  Analogously, since $f(\M)\subset \bar{\M}$ is a   codimension two submanifold, the term $\contr(\e\otimes \bar{g}^{\alpha_1 \beta_1}\otimes g^{\alpha_2 \beta_2} \otimes R^\perp_{\alpha_3 \beta_3 \alpha_4 \beta_4}) $ 
is a (possibly null) multiple of the normal curvature $K^\perp$.  But,  by  \eqref{eq:TopInv},  we already know that $\int_{\M} K d \mu_g$ and $\int_{\M} K^\perp d \mu_g$ are global conformal invariants so, called
\begin{eqnarray}
P_1(g, h)&:=& P(\bar{g},R,R^\perp, h)- a K- bK^\perp \nonumber \\
                        &=&  b \contr(g^{i_1j_1}\otimes g^{i_2 j_2}\otimes H g_{i_{3} j_{3}}\otimes \ldots \otimes H g_{i_{r+2} j_{r+2}}\otimes h^\circ_{i_{r+3} j_{r+3}}\otimes \ldots \otimes  h^\circ_{i_{r+4} j_{r+4}}  ), \nonumber
\end{eqnarray}
and so $\int_{\M} P_1(g,h) d \mu_g$ is a global conformal invariant for compact surfaces immersed in 3-manifolds, as the difference of such objects. The thesis can be now achieved by repeating verbatim the second part of the proof of  Theorem \ref{thm:CD1surf}.
\end{proof}

\subsubsection{Two examples: the complex projective plane and the complex hyperbolic plane} 
Two particular cases of interest (apart from the spaces forms) are $ \mbb{CP } ^ 2 $ and $ \mbb{CH}^  2$ the complex projective plane and the complex hyperbolic plane respectively. These are K\"ahler manifolds with their standard K\"ahler form $ \Omega$ of constant holomorphic sectional curvature and unlike their real counterparts, $\mbb{S}^4$ and $ \mbb H ^ 4$, they are not locally conformal to $ \mbb {C} ^ 2 $.

Let us consider an immersion $ \phi : \M \hookrightarrow \bar{\M}$ of an oriented  surface, where $ \bar{\M} = \mbb {CP}^ 2, \mbb{CH}^2$. The K\"ahler function $C$ on $\M $ is defined by $ \phi ^ * \Omega = C d \mu_g$. Only the sign of $ C$ depends on the orientation, hence $ C ^2 $ and $ |C|$ are well defined for non-orientable surfaces. 
The K\"ahler function satisfies 
\begin{align*}
-1 \leq C \leq 1 \quad.
\end{align*} 
By direct computation, we  find that the Willmore functional is equal to
\begin{align*}
\mc {W}_{\mbb{CP}^ 2 }( \phi) = \int _{\M}  \left( | H | ^ 2 + \ov K \right) \, d \mu_g = \int _{ \M} \left( | H| ^ 2 + 1 + 3 C ^ 2 \right) d \mu_g \quad ,
\end{align*}
and 
\begin{align*}
\mc{W}_{\mbb{CH} ^ 2 } ( \phi )= \int _{\M} \left( | H | ^ 2 + \ov K \right) d \mu_g = \int _{ \M} \left( | H| ^ 2 - 1 -3 C ^ 2 \right) d \mu_g \quad.
\end{align*}
By the Ricci equation \eqref{eq:Ricci}, we also have
\begin{align*}
\ov K ^ \perp = \ov  R_{1234} = K^\perp - \left(( \circo h_{1p})_{3}  (\circo h _{ 2p})_{4} - (  \circo h _{ 2 p})_{ 3 }( \circo  h _{ 1 p})_{4 } \right)
\end{align*}
which, applying the symmetries of the curvature tensor, can be written as a complete contraction.  Therefore, the following energies are all global conformal invariant:

\begin{align} 
\mc W ^ +_{\mbb {CP}^2} ( \phi) &= \int _{ \M} | H| ^ 2 + \ov K - \ov K ^ \perp = \int_{\M} (|H|^ 2 + 6 C ^ 2 ) d \mu_g \label{eq:WCP21} \\ 
\mc W ^ -_{\mbb {CP}^2} ( \phi) &= \int _{ \M} | H| ^ 2 + \ov K + \ov K ^ \perp = \int_{\M} (|H|^ 2 + 2) d \mu_g \quad,   \label{eq:WCP22}
\end{align}
and 
\begin{align}
\mc W ^ +_{\mbb {CH}^2} ( \phi) &= \int _{ \M} | H| ^ 2 + \ov K - \ov K ^ \perp = \int_{\M} (|H|^ 2 - 6 C ^ 2 ) d \mu_g  \label{eq:WCH21}  \\
\mc W ^ -_{\mbb {CH}^2} ( \phi) &= \int _{ \M} | H| ^ 2 + \ov K + \ov K ^ \perp = \int_{\M} (|H|^ 2 - 2) d \mu_g \quad.  \label{eq:WCH22}
\end{align}
Let us remark that the energies \eqref{eq:WCP21},\eqref{eq:WCP22}  have already been object of investigation in  \cite{MontielUrbano}, where it was shown that  $ \mc {W}^-_{\mbb {CP} ^  2} ( \phi) \geq 4 \pi \mu - 2 \int_{\Sigma}|C| d \mu_g  $ where $\mu$ is the maximum multiplicity of the immersion (i.e. the maximal number of pre-images via the immersion map). Moreover in the same paper it was shown that equality holds if and only if $ \mu =1 $ and $\M$ is either a complex projective line or a totally geodesic real projective plane,  or $ \mu = 2 $ and $ \M$ is a Lagrangian Whitney sphere.

\begin{remark}
Since the goal of the present paper is to investigate the structure of global conformal invariants, we recalled the definition of  $\mc W ^{\pm}_{\mbb {CP}^2}$ given in   \cite{MontielUrbano} and we defined the new functionals $\mc W ^{\pm}_{\mbb {CH}^2}$ in order to give interesting examples of Willmore-type energies in codimension 2. In a forthcoming work we will address the question wether    $\mc W ^{\pm}_{\mbb {CH}^2}$ satisfy analogous properties as  $\mc W ^{\pm}_{\mbb {CP}^2}$. 
\end{remark}

\section{Global Conformal Invariants of Submanifolds} 
Let us consider a geometric scalar  quantity $ P (g^m, h^m)$ of the form 
\begin{align} \label{eqn_star}
\sum_{ l \in L } a _l C^ l ( g^m , h^m) \quad,
\end{align}
 where each $C  ^ l$ is a complete contraction 
\begin{align} \label{eqn_contract}
\contr(g^{i_1 j_1}\otimes\ldots \otimes  g^{i_s j_s} \otimes h_{ i _{s+1}  j_{s+1}}\otimes \dots \otimes h _{ i _{2s} j_{2s}}) \quad ; 
\end{align}
that is  we consider complete contractions as defined in  \eqref{eq:ContrGen} but depending just on the second fundamental form $h^m$ and the induced metric $g^m$ for immersed $m$-submanifolds $f(\M^m)$ in Riemannian $n$-manifolds $(\bar{\M}^n,\bar{g}^n)$. 

Our first  goal in this section  is to understand the structure of geometric scalar quantities  \eqref{eqn_star} giving rise to global conformal invariants for submanifolds, in the sense of Definition \ref{def:GCI}.
This is exactly  the content of the next result.
  
\begin{theorem}\label{thm:P(g,h)}
Let $ P ( g^m ,h^m  ) $ be as in  \eqref{eqn_star} with each $ C^l$ of the form \eqref{eqn_contract}, and  assume that $ \int _{\M}P( g ^ m ,h ^ m) d \mu_{g^m}$ is a global conformal invariant, in the sense of Definition \ref{def:GCI}. Then there exists a  \emph{pointwise} conformal invariant $W ( g ^ m, \circo h^m )$ of weight $ -m$  depending only on the traceless second fundamental form $\circo h$ contracted with the induced metric $g^m$ so that 
\begin{align*}
P( g ^ m , h ^ m ) = W ( g ^ m,  \circo h ^m )\quad .
\end{align*} 
In other words, for every $l \in L$ one has that  $ C ^l$ in \eqref{eqn_star} is a complete contraction of  weight $-m$ of  the form 
\begin{align} \label{eqn_contract}
\contr(g^{i_1 j_1}\otimes\ldots \otimes  g^{i_m j_m} \otimes \circo h_{ i _{m+1}  j_{m+1}}\otimes \dots \otimes \circo h _{ i _{2m} j_{2m}}) \quad . 
\end{align}

\end{theorem}

\begin{proof}
First of all  recall that, by definition, $P ( g^m ,h^m  ) $ is a linear combination of complete contractions  $C^ l ( g^m , h^m)$ each  of weight $-m$. Observing that  any contraction of $g^{-1} \otimes h$ is of weight $-1$, this implies that in \eqref{eqn_contract} we have $s=m$, that is $C  ^ l(g^m, h^m)$ is a complete contraction of the form 
\begin{align} \label{eqn_contract1}
\contr(g^{i_1 j_1}\otimes\ldots \otimes  g^{i_m j_m} \otimes h_{ i _{m+1}  j_{m+1}}\otimes \dots \otimes h _{ i _{2m} j_{2m}}) \quad .
\end{align}
Recalling \eqref{eq:Splith}, that is the orthogonal splitting  $h_{ij}=H g_{ij}+h _{ij}^\circ$ , we can rewrite such a complete contraction as
\begin{align} \label{eqn_contract1}
\contr(g^{i_1 j_1}\otimes\ldots \otimes  g^{i_m j_m} \otimes  H g_{ i _{m+1}  j_{m+1}}\otimes \dots \otimes H g_{ i_{m+r} j_{m+r}} \otimes  h_{ i _{m+r+1}  j_{m+r+1}}^\circ  \otimes \dots \otimes  h _{ i _{2m} j_{2m}}^\circ )  \quad .
\end{align}
Our goal is to prove that $r=0$, that is there are no $H$ factors, so $P(g^m,h^m)=P(g^m,\circo h^m)$ is expressed purely as complete contractions of traceless fundamental forms, which are  \emph{pointwise} conformal invariants once multiplied by $d\mu_g$ thanks to \eqref{eq:CChH}.\\

To that aim  consider an arbitrary compact $m$-dimensional immersed submanifold $f(\M^m)$ of an arbitrary  Riemannian $n$-manifold $(\bar{\M}^n, \bar{g}^n)$, and an arbitrary conformal rescaling $\hat{\bar{g}}^n=e^{2 \phi(x)} \bar{g}^n$ of the ambient metric by a smooth function $\phi\in C^\infty(\bar{\M})$. All the hatted geometric tensors $\hat{g}, \hat{h}, \hat{h}^\circ, \hat H$ denote the corresponding tensors computed with respect to the deformed metric $\hat{\bar{g}}$. In Section \ref{SS:Ig} we defined the operator
$$I_{(g,h)}(\phi):= e^{m \phi} P(\hat{g}, \hat{h})- P(g,h)\quad ,$$
and we observed that  the conformal invariance of the integrated quantities implies (see \eqref{eq:IntIInv}) 
\begin{equation}\label{eq:IntI=0}
\int_{\M} I_{(g,h)}(\phi) \, d \mu_{g}=0 \quad.
\end{equation}
From the formulas \eqref{eq:CChH} of the change of $h^\circ$ and $H$ under a conformal deformation of metric, it is clear that $I_{(g,h)}(\phi)$ does not depend directly on $\phi$ but just on $\pi_N(\bar{\nabla} \phi)$, the projection of $\bar{\nabla} \phi$ onto the normal space of $f(\M)$. More precisely  $I_{(g,h)}(\phi)$ is  polynomial in the components of $\pi_N(\bar{\nabla} \phi)$ exactly of  the same degree $0\leq r\leq m$ as $P(g, h^\circ, H)$ seen as a polynomial in $H$.  

By considering $t\phi$ for $t \in \R$, we get that 
$$I_{g, h}(t \phi)= \sum_{k=1}^r a_k C^k(g, h^\circ, H, \pi_N(\bar{\nabla} \phi)) \; t^k\quad ,$$
where $C^k(g, h^\circ, H, \pi_N(\bar{\nabla} \phi))$ is an homogeneous polynomial of degree $k$ in the components of  $\pi_N(\bar{\nabla} \phi)$. Recalling now \eqref{eq:IntI=0}, we obtain that  $ \int_{\M} I_{(g,h)}(t\phi) \, d \mu_{g} $ vanishes identically as a polynomial in $t$, so
\begin{equation}\label{eq:IntIk}
0=\frac{d^k}{dt^k}\mid_{{t=0}}  \int_{\M} I_{(g,h)}(t\phi) \, d \mu_{g}  = k!  \int_{\M}  a_k C^k(g, h^\circ, H, \pi_N(\bar{\nabla} \phi))   \, d \mu_{g},\quad \forall k=0,\ldots,r.
\end{equation}
Pick an arbitrary point $x \in \M$; by choosing local coordinates in $\bar{\M}^n$ adapted to $f(\M)$ at $f(x)$, it is easy to see that for any given functions $\psi^i,\ldots, \psi^{n-m} \in C^{\infty}_c(\M)$ supported in such a coordinate neighborhood  of $x$, there exists $\phi \in C^\infty(\bar{\M})$ such that $$\psi^i= (\bar{\nabla} \phi)^{m+i} \circ f \quad, $$
 where of course thanks to this choice of coordinates we have 
 $$\pi_N(\bar{\nabla} \phi)=\Big((\bar{\nabla} \phi)^{m+1}, \ldots, (\bar{\nabla} \phi)^{n}\Big)\quad. $$
 By plugging such arbitrary $C^\infty_c(\M)$ functions   $\psi^i,\ldots, \psi^{n-m}$ in place of  $\pi_N(\bar{\nabla} \phi)$ in \eqref{eq:IntIk}, we obtain that not only the integrals  but the integrands themselves must vanish, that is $a_k C^k(g, h^\circ, H, \pi_N(\bar{\nabla} \phi))\equiv 0$ on $\M$. It follows that $I_{g,h}(\phi)\equiv 0$ or, in other words, the degree of  $I_{g, h}(t \phi)$ as a polynomial in $t$ is $0$. By the above discussion we have that $r=0$, which was our thesis. \end{proof}
 
We pass now to  consider the more general geometric scalar  quantity $ P (g^m, R^m, h^m)$, for $m\in \N$ \emph{even}, of the form 
\begin{equation}\label{eq:PP1P2}
P (g^m, R^m, h^m)=P_1(g^m, h^m)+ P_2(g^m,R^m)\quad,
\end{equation}
where
\begin{align} \label{eq:P1P2}
P_1(g^m,h^m)=\sum_{ l \in L } a _l C^ l ( g^m , h^m) \quad \text{and} \quad  P_2(g^m,R^m)=\sum_{ s \in S } b _s C^ s ( g^m , R^m) \quad,
\end{align}
 where each $C  ^ l( g^m , h^m)$ is a complete contraction 
\begin{align} \label{eq:Cl}
\contr(g^{i_1 j_1}\otimes\ldots \otimes  g^{i_s j_s} \otimes h_{ i _{s+1}  j_{s+1}}\otimes \dots \otimes h _{ i _{2s} j_{2s}}) \quad 
\end{align}
and each  $C^ s ( g^m , R^m)$ is a complete contraction
\begin{align} \label{eq:Cs}
\contr(g^{i_1 j_1}\otimes\ldots \otimes  g^{i_{2r} j_{2r}} \otimes R_{ i _{2r+1}  j_{2r+1} k_{2r+1} l_{2r+1} }\otimes \dots \otimes R_{ i _{3r}  j_{3r} k_{3r} l_{3r} }) \quad. 
\end{align}
In other words  we consider complete contractions as defined in  \eqref{eq:ContrGen} which  split in two parts: one depending just on the second fundamental form $h^m$ and the other one just on the intrinsic curvature $R^m$, 
for immersed $m$-submanifolds $(f(\M^m), g^m)$ in Riemannian $n$-manifolds $(\bar{\M}^n,\bar{g}^n)$. 
As usual, the goal is to understand the structure of geometric scalar quantities  \eqref{eq:PP1P2} giving rise to global conformal invariants for submanifolds, in the sense of Definition \ref{def:GCI}.
This is exactly  the content of the next result.

\begin{theorem}\label{thm:P=P1+P2}
Let $m \in \N$ be  \emph{even} and let  $P( g ^ m , R^m, h ^m )=P_1(g^m, h^m)+ P_2(g^m,R^m)$ be a geometric scalar quantity as above. Assume that  $ \int _{\M}P( g ^ m, R^m, h ^ m) d \mu_{g^m}$ is a global conformal invariant, in the sense of Definition \ref{def:GCI}. 

Then both $ \int _{\M}P_1( g ^ m, h ^ m) d \mu_{g^m}$ and $ \int _{\M}P_2( g ^ m, R ^ m) d \mu_{g^m}$ are global conformal invariants. It follows that

 i)  There exists a  \emph{pointwise} conformal invariant $W_1 ( g ^ m, \circo h^m )$ of weight $ -m$  depending only on the traceless second fundamental form $\circo h$ contracted with the induced metric $g^m$ so that 
\begin{align*}
P_1( g ^ m , h ^ m ) = W_1 ( g ^ m,  \circo h ^m )\quad ;
\end{align*} 
or, in other words, for every $l \in L$ one has that  $ C ^l$ in \eqref{eq:Cl} is a complete contraction of  weight $-m$ of  the form 
\begin{align} \label{eqn_contract2}
\contr(g^{i_1 j_1}\otimes\ldots \otimes  g^{i_m j_m} \otimes \circo h_{ i _{m+1}  j_{m+1}}\otimes \dots \otimes \circo h _{ i _{2m} j_{2m}}) \quad . 
\end{align}

ii) Called  $ \Pfaff(R^m)$  the Pfaffian of the intrinsic Riemann tensor $R^m$ and $W^m$ the Weyl tensor of $g^m$,  $P_2(g^m, R^m)$ is of the form
\begin{align*}
P_2(g^m, R^m)= \tilde{P}_2(g^m, W^m)+ c \Pfaff(R^m) \;, \quad \text{for some } c\in \R \quad, 
\end{align*}
where $\tilde{P}_2(g^m, W^m)$ is a \emph{pointwise} conformal invariant of weight $-m$ expressed as a  linear combination of complete contractions of the form
$$\contr(g^{i_1 j_1}\otimes\ldots \otimes  g^{i_{m} j_{m}} \otimes W_{ i _{m+1}  j_{m+1} k_{m+1} l_{m+1} }\otimes \dots \otimes W_{ i _{\frac{3m}{2}}  j_{\frac{3m}{2}} k_{\frac{3m}{2}} l_{\frac{3m}{2}} }) \quad.  $$
\end{theorem}

\begin{remark}
It is well known that 
\begin{itemize}
\item The Weyl tensor $W_{ijkl}(g^m)$  is a pointwise scalar conformal invariant of weight $2$, that is it satisfies $W_{ijkl}(e^{2\phi(x)} g^m(x))= e^{2 \phi(x) } W(g^m)(x)$ for every $\phi\in C^\infty(\M)$ and every $x\in \M^m$. It follows that any complete contraction of the tensor $g^{-1}\otimes g^{-1} \otimes W$ is a pointwise scalar conformal invariant of weight $-2$.
\item The Pfaffian  $\Pfaff(R_{ijkl})$ of the intrinsic curvature $R_{ijkl}=R^m$ integrated over the manifold gives rise to a \emph{topological} invariant:
$$\int_{\M^m} \Pfaff(R_{ijkl}) \, d\mu_g= \frac{2^m \pi^{m/2} (\frac{m}{2}-1)!}{2(m-1)!} \chi(\M^m)\quad, $$
where $\chi(\M^m)$ is the Euler Characteristic of $\M^m$.
\end{itemize}
Therefore, recalling the conformal invariance of the traceless second fundamental form \eqref{eq:CChH}, any linear combination of complete contractions
$$P(g^m,h^m,R^m)=W_1(g^m, \circo h^m)+ \tilde{P}_2(g^m, W^m)+ c \Pfaff(R^m)$$
as in the thesis of Theorem \ref{thm:P=P1+P2} gives rise to a ``trivial'' global conformal invariant. Thereom \ref{thm:P=P1+P2} states  that, under the assumption that $P$ splits into the sum of an intrinsic part and of an extrinsic part depending just on the second fundamental form, this is actually the only possibility.
 \end{remark}
 
 \begin{proof}
 Since by assumption $P( g ^ m , R^m, h ^m )=P_1(g^m, h^m)+ P_2(g^m,R^m)$ gives rise to a global conformal invariant, if we show that $\int_{\M} P_1(g^m, h^m) \, d \mu_g$ is a global conformal invariant, the same will be true for   $\int_{\M} P_2(g^m, R^m) \, d \mu_g$.    In order to prove that, consider  an arbitrary compact $m$-dimensional immersed submanifold $f(\M^m)$ of an arbitrary  Riemannian $n$-manifold $(\bar{\M}^n, \bar{g}^n)$, and an arbitrary conformal rescaling $\hat{\bar{g}}^n=e^{2 \phi(x)} \bar{g}^n$ of the ambient metric by a smooth function $\phi\in C^\infty(\bar{\M})$. All the hatted geometric tensors $\hat{g}, \hat{h}, \hat{h}^\circ, \hat H, \hat{R}$ denote the corresponding tensors computed with respect to the rescaled metric $\hat{\bar{g}}$. Analogously to  Section \ref{SS:Ig}, define the operators
 \begin{eqnarray}
 I^1_{g,h}(\phi)&:=& e^{m \phi} P_1(\hat{g}, \hat{h})- P_1(g,h)\quad , \label{eq:I1} \\
 I^2_{g,R}(\phi)&:=& e^{m \phi} P_2(\hat{g}, \hat{R})- P_2(g,R)\quad , \label{eq:I2}   \\
 I_{g,R,h}(\phi)&:=& I^1_{g,h}(\phi)+I^2_{g,R}(\phi) \quad . \label{eq:I=I1+I2}
 \end{eqnarray}
As already observed in the proof of Theorem \ref{thm:P(g,h)},   $I^1_{g,h}(\phi)$ does not depend directly on $\phi$ but  only on $\pi_N(\bar{\nabla} \phi)$, the projection of $\bar{\nabla} \phi$ onto the normal space of $f(\M)$. More precisely  $I^1_{g,h}(\phi)$ is  polynomial in the components of $\pi_N(\bar{\nabla} \phi)$ exactly of  the same degree $0\leq r\leq m$ as $P_1(g, h^\circ, H)$ seen as a polynomial in $H$. 
\\On the other hand, recalling the formulas \eqref{eq:Rhatg} of the change of the intrinsic curvature $R_{ijkl}$ of $(\M^m, g^m)$ under conformal deformation of the metric, it is clear that $I^2_{g,R}$ does not depend directly on $\phi$ but only on $\nabla_\M \phi$, the projection of $\bar{\nabla} \phi$ onto the tangent  space of $f(\M)$, and on $\nabla^2_\M \phi$, the covariant Hessian of $\phi|_{\M}$.  More precisely  $I^2_{g,R}(\phi)$ is  polynomial in the components of $\nabla_\M \phi, \nabla^2 _\M \phi$.
 
By considering $t\phi$ for $t \in \R$, we get that 
$$I^1_{g, h}(t \phi)= \sum_{k=1}^r a_k C^k(g, h^\circ, H, \pi_N(\bar{\nabla} \phi)) \; t^k\quad  ,\quad I^2_{g,R}(t \phi)= \sum_{l=1}^s b_l C^l(g, R,  \nabla_\M \phi, \nabla^2 _{\M} \phi) \; t^k\quad $$
where $C^k(g, h^\circ, H, \pi_N(\bar{\nabla} \phi))$ (respectively $C^l(g, R,  \nabla_\M \phi, \nabla^2 _{\M} \phi)$) is an homogeneous polynomial of degree $k$ in the components of  $\pi_N(\bar{\nabla} \phi)$ (respectively of $\nabla_\M \phi, \nabla^2 _\M \phi$). Recalling now \eqref{eq:IntI=0}, we obtain that  $ \int_{\M} I_{(g,h)}(t\phi) \, d \mu_{g} $ vanishes identically as a polynomial in $t$, so
\begin{eqnarray}
0&=&\frac{d^k}{dt^k}\bigg|_{{t=0}}  \int_{\M} I_{(g,h)}(t\phi) \, d \mu_{g}  \nonumber \\
   &=& k!  \int_{\M}  a_k C^k(g, h^\circ, H, \pi_N(\bar{\nabla} \phi)) + b_k C^k(g, R,  \nabla_\M \phi, \nabla^2 _{\M} \phi)    \, d \mu_{g},\quad \forall k\in \N. \label{eq:akbk}
\end{eqnarray}
Pick an arbitrary point $x \in \M$; by choosing local coordinates in $\bar{\M}^n$ adapted to $f(\M)$ at $f(x)$, it is easy to see that for any given functions $\psi^i,\ldots, \psi^{n-m} \in C^{\infty}_c(\M)$ supported in such coordinate neighborhood  of $x$, there exists $\phi \in C^\infty(\bar{\M})$ such that 
$$\psi^i= (\bar{\nabla} \phi)^{m+i} \circ f \quad \text{and}  \quad \nabla_{\M} (\phi\circ f)=0 \quad,  $$
 where of course  thanks to the this choice of coordinates we have 
 $$\pi_N(\bar{\nabla} \phi)=\Big((\bar{\nabla} \phi)^{m+1}, \ldots, (\bar{\nabla} \phi)^{n}\Big)\quad. $$
 With this choice of $\phi$, the second summand in the integral of \eqref{eq:akbk} disappears, and thanks to the arbitrariness of the  $C^\infty_c(\M)$ functions   $\psi^i,\ldots, \psi^{n-m}$ in place of  $\pi_N(\bar{\nabla} \phi)$ in \eqref{eq:akbk}, we obtain that not only the integrals  but the integrands themselves must vanish, that is $a_k C^k(g, h^\circ, H, \pi_N(\bar{\nabla} \phi))\equiv 0$ on $\M$, so $I^1_{g,h}(\phi)\equiv 0$.  
 In particular $\int_{\M} P_1(g^m,h^m) d \mu_g$ is a global conformal invariant for $m$-dimensional submanifolds which implies that also $\int_{\M} P_2(g^m,R^m) d \mu_g$ is a global conformal invariant, since by assumption $\int_{\M}  [P_1(g^m,h^m) +  P_2(g^m,R^m) ] d \mu_g $ is that as well.
 
Claim i) follows then directly from Theorem \ref{thm:P(g,h)}.  To get claim ii) observe that by construction  $P_2(g^m, R^m)$ depends just on the \emph{intrinsic} Riemannian structure $(\M,g)$ and not on the immersion $f$ into an ambient manifold $\bar{\M}$. Therefore we proved that $\int_{\M} P_2(g^m, R^m) d \mu_g$ is an intrinsic Riemannian conformally invariant quantity, which enters into the framework of the papers of Alexakis \cite{AlexI}, \cite{AlexII}. More precisely, by applying \cite[Theorem 1]{AlexI}, we obtain claim ii) and the proof is complete.   
  \end{proof}

\section{Generalized Willmore Energies in Higher Dimensions}\label{sec:genWill}
In this final section,  we will introduce a higher dimensional analogue of the Willmore energy (actually we will construct a two-parameters family of such functionals). This new energy is conformally invariant and only attains its strictly positive lower bound at a round sphere. Let us start with some preliminaries about the Willmore functional. 

Given an immersion $f:\M^2 \hookrightarrow \R^3$ of a closed surface $\M^2$ into the Euclidean space $\R^3$, the Willmore functional $\mc W(f)$ is defined by
$$\mc W (f):=\int_{\M} |H|^2 d \mu_g\quad. $$
A natural way to introduce such functional is via conformal invariance: by the Gauss-Bonnet Theorem,  $\int_{\M} K d \mu_g$ is a topological hence a fortiori conformal invariant quantity; moreover, by the formula \eqref{eq:CChH}, $|h^\circ|^2 d \mu_g$ is a  \emph{pointwise} conformal invariant. It follows that \emph{any} functional 
$$ W_{\alpha}(f):=\int_{\M} \left( K + \alpha  |h^\circ|^2  \right) d \mu_g  \text{ is a global conformally invariant 	quantity}$$
in the sense of Definition  \ref{def:GCI}.
For immersions into $\R^3$ we have that    $K+\frac{1}{2} |h^\circ|^2= |H|^2  $, so 
$$\mc W(f) =  W_{1/2}(f)   \text{ is invariant under conformal trasformations of $\R^3$,} $$
that is under  Moebius transformations centred out of $f(\M)$. 
Observing that $K=\detg(h)$ and that $\frac{1}{2}|h^\circ|^2=-\detg(h^\circ)$, we can also write
\begin{equation}\label{eq:Wdeth}
\mc W(f)=\int_{\M} \left( \detg(h) - \detg(h^\circ) \right) \, d \mu_g \quad.
\end{equation}
Of course the functional $\mc W$ is non-negative and vanishes exactly on minimal surfaces, which are therefore points of strict global minimum; the critical points of $\mc W$ can therefore be seen in a natural way as ``generalized conformal minimal surfaces''; this was indeed the starting point  in the '20ies of the theory of Willmore surfaces by Blaschke \cite{Blaschke},  who was looking for a natural conformally invariant class of immersions which included minimal surfaces. Let us mention that such a functional was later rediscovered in the 1960's by Willmore \cite{Will} who proved that round 2-spheres
are the points of strict global minimum for $\mc W$.

Motivated by this celebrated  two dimensional theory, our goal is to investigate the case of 4-d hypersurfaces in $\R^5$, that is $f:(\M^4,g)\hookrightarrow (\R^5, \delta_{\mu\nu})$ isometric immersion. \\\\We address the following natural questions:

1) Is it possible to ``perturb'' the Pfaffian of the Riemann tensor of the induced metric $g$  on $\M^4$ in order to get a conformally invariant functional vanishing on minimal surfaces? 

2) Is that functional positive definite? If not, how can we preserve the conformal invariance and make it positive definite?

3) Are round spheres of $\R^5$ strict global minimum of this conformally invariant functional?
\medskip

As in the 2-d case, the starting point is the Gauss-Bonnet Theorem. For $4$-d smooth closed (that is compact without boundary) immersed hypersurfaces $f:(\M^4,g) \hookrightarrow (\R^5, \delta_{\mu \nu})$, the Gauss-Bonnet Theorem states that    
\begin{equation}\label{eq:GenGB}
\int_{\M^4} \detg(h) \, d \mu_g=\frac{4 \pi^2 }{3} \; \chi(\M) = \frac{8 \pi^2 }{3} \deg(\gamma) \quad,
\end{equation}    
where $\detg(h):=\det(g^{-1} h)$, $\chi(\M)$ is the Euler Characteristic of $\M$ and  $\gamma: \M^4 \to S^4\subset \R^5$ is the Gauss map associated to  the immersion $f$.
By applying the classical Newton's identities for symmetric polynomials to the symmetric polynomials of the principal curvatures of the immersion $f$, it is an easy exercise to write   $\det(g^{-1} h)$ as
\begin{equation}\label{eq:detgh}
\detg(h)=\frac{32}{3} H^4- 4 H^2 |h|^2+\frac{4}{3} H \Tr_g(h^3)+\frac{1}{8}|h|^4-\frac{1}{4} \Tr_g(h^4) \quad,
\end{equation}
where  $\Tr_g(h^p):=\Tr[(g^{-1} h)^p]$. By using the orthogonal decomposition $h=\circo h+ H g$, see \eqref{eq:Splith},  iteratively we get that
\begin{eqnarray}
|h|^2&=&  \Tr_g(h^2) = |h^\circ|^2 + 4 H^2 \nonumber \\
\Tr_{g}(h^3)&=& \Tr_{g}(\circo h^3)+ 3 H |\circo h|^2 + 4 H^3 \nonumber\\
\Tr_{g}(h^4)&=& \Tr_{g}(\circo h^4)+ 4 H \Tr_{g}(\circo h^3) + 6 H^2 |\circo h|^2 + 4 H^4 \quad, \nonumber
 \end{eqnarray}
which, plugged into \eqref{eq:detgh}, give
\begin{equation}\label{eq:detg-1h}
\detg(h)= H^4- \frac{1}{2} H^2 |\circo h|^2+\frac{1}{3} H \Tr_g(\circo h ^3)+\frac{1}{8}|\circo h|^4-\frac{1}{4} \Tr_g(\circo h^4) \quad.
\end{equation}
Inspired by the 2-d case, in particular by formula \eqref{eq:Wdeth}, we wish to make appear a term $\detg(\circo h)$, this is possible thanks to the following lemma.

\begin{lemma}\label{lem:Trdet}
For any  $4$-d hypersurface immersed into $\R^5$, we have that
\begin{equation}\label{eq:Trdet}
\Tr_g(\circo h^4)-\frac{1}{2} |\circo h|^4 = -4 \detg(\circo h) \quad.
\end{equation} 
\end{lemma}

\begin{proof}
Since $\circo h$ is bilinear symmetric with respect to $g$ it can be diagonalized and its eigenvalues $\{\lambda_i\}_{i=1,\ldots,4} \subset \R$ satisfy $\lambda_4= -\lambda_1-\lambda_2-\lambda_3$, since $\Tr_g(\circo h)=0$. We then have that 
\begin{eqnarray}
\Tr_g(\circo h^4)& = &\lambda_1^4+ \lambda_2^4 + \lambda_3^4 + (\lambda_1+\lambda_2+\lambda_3)^4 \nonumber\\
                          &=& 2 \sum_{i=1}^3 \lambda_i^4 + 4 \sum_{1\leq i\neq j \leq 3} \lambda_i^3 \lambda_j + 6 \sum_{1\leq i\neq j \leq 3} \lambda_i^2 \lambda_j^2 + 12 \sum_{1\leq i\neq j \neq k \leq 3} \lambda_i^2 \lambda_j \lambda_k \quad.\label{eq:Trh4} 
\end{eqnarray}
On the other hand,
\begin{eqnarray}
|\circo h|^4 & = & \left[\lambda_1^2+ \lambda_2^2 + \lambda_3^2 + (\lambda_1+\lambda_2+\lambda_3)^2\right]^2 \nonumber\\
                          &=& 4 \sum_{i=1}^3 \lambda_i^4 + 8 \sum_{1\leq i\neq j \leq 3} \lambda_i^3 \lambda_j + 12 \sum_{1\leq i\neq j \leq 3} \lambda_i^2 \lambda_j^2 + 16 \sum_{1\leq i\neq j \neq k \leq 3} \lambda_i^2 \lambda_j \lambda_k \quad.\label{eq:h4} 
\end{eqnarray}
Combining  \eqref{eq:Trh4}  and \eqref{eq:h4}  and recalling that  $\lambda_4= -\lambda_1-\lambda_2-\lambda_3$ we can conclude: 
$$
\Tr_g(\circo h^4)- \frac{1}{2}|\circo h|^4 = 4 \sum_{1\leq i\neq j \neq k \leq 3} \lambda_i^2 \lambda_j \lambda_k  = 4 \lambda_1 \lambda_2 \lambda_3 (\lambda_1+ \lambda_2+\lambda_3)= -4 \detg(\circo h) \quad.
$$
\end{proof}
Thanks to the identity \eqref{eq:Trdet}, we can rewrite \eqref{eq:detg-1h} as
\begin{equation}\label{eq:dethH}
\detg(h)- \detg(\circo h)= H^4- \frac{1}{2} H^2 |\circo h|^2+\frac{1}{3} H \Tr_g(\circo h ^3)\quad.
\end{equation}
 
Since for even dimensional hypersurfaces in the Euclidean space it is well know that the Pfaffian of the intrinsic Riemann tensor is a multiple of $\detg(h)$, we  have just answered to question 1):

\begin{proposition}\label{prop:1}
Given an isometric  immersion $f:(\M^4,g)\hookrightarrow (\R^5, \delta_{\mu\nu})$ of a closed $4$-manifold  $(\M^4,g)$, define
$$\mc P (f):= \int_{\M} \left[ \detg(h)- \detg(\circo h) \right] d \mu_g \quad. $$
Then, the  functional $\mc P$ is invariant under conformal transformations of $\R^5$ centered off of $f(\M^4)$. 
\\Moreover $\mc P $ vanishes identically on minimal hypersurfaces, and the minimal hypersurfaces of $\R^5$ satisfying $\Tr_g(\circo h^3)\equiv 0$ are critical points  for $\mc P$.
\end{proposition}

\begin{proof}
By using \eqref{eq:dethH},  we have that $\detg(h)- \detg(\circo h) $
 clearly vanishes if $H\equiv 0$. Observe that the right hand side of \eqref{eq:dethH} has a linear term in $H$, so a priori a minimal hypersurface is  not a critical point for $\mc P$, but if we assume also that  $\Tr_g(\circo h ^3)$ vanishes, of course we  obtain criticality. Since $\left[ \frac{1}{4} \Tr_g(\circo h^4)- \frac{1}{8}|\circo h|^4 \right] d \mu_g$ is a pointwise conformal invariant and  $\detg(h)$ is a topological invariant by \eqref{eq:GenGB}, the functional $\mc P$ is then invariant under conformal transformation of $\R^5$ preserving the topology of $f(\M)$, that is under conformal transformations of $\R^5$ centered off of $f(\M^4)$.
\end{proof}

\begin{remark}
Notice that, formally,  the integrand of the 4-d functional $\mc P$ is exactly the same as the 2-d Willmore functional $\mc W$, written as in \eqref{eq:Wdeth}.
\end{remark}

Since in 2-d the quantity $\detg(h)-\detg(h)=H^2$ is non-negative, it is natural to ask if the same is true in the 4-d case, that is if $\detg(g)-\detg(\circo h)$ is non negative.   Surprisingly this is not the case, for instance it is not difficult to compute that if the principal curvatures are $1,1,6,6,$ one has $\detg(h)-\detg(\circo h)=-\frac{49}{16}<0$. 
A natural question is then if we can manipulate $\mc P$ it in order to obtain a new functional which is still conformally invariant but this time is nonnegative definite; this is exactly question 2) above. To this aim observe that, by Young's inequality, we have that
\begin{eqnarray}
-\frac{1}{2} H^2 |\circo h|^2 &\geq& -\frac{\alpha}{4} H^4 - \frac{1}{4 \alpha} |\circo h|^4 \quad \text{for every } \alpha>0  \quad\text{and }\nonumber \\
\frac{1}{3} H \Tr_g(\circo h ^3) &\geq&  -\frac{\beta}{12} H^4 - \frac{1}{4 \beta^{1/3}}  [\Tr_g(\circo h ^3)]^{4/3}   \quad \text{for every } \beta>0 \quad. \nonumber
\end{eqnarray}
The last two lower bounds combined with \eqref{eq:detg-1h} give
\begin{equation}\label{eq:Pab}
P_{(\alpha, \beta)} (g,h):= \detg(h)+ \frac{1}{4} \Tr_g(\circo h^4)+\left(\frac{1}{4 \alpha} -\frac{1}{8} \right) |\circo h|^4 + \frac{1}{4 \beta^{1/3}}  [\Tr_g(\circo h ^3)]^{4/3}  \geq \left(1- \frac{3\alpha+\beta}{12} \right) H^4.  
\end{equation}
We have therefore answered  question 2) above:

\begin{proposition}\label{prop:2}
Given an isometric  immersion $f:(\M^4,g)\hookrightarrow (\R^5, \delta_{\mu\nu})$ of a closed $4$-manifold  $(\M^4,g)$,  let $ P_{(\alpha, \beta)}$ be the following expression
\begin{equation}\label{eq:defPab}
P_{(\alpha, \beta)} (g,h):= \detg(h)+ \frac{1}{4} \Tr_g(\circo h^4)+\left(\frac{1}{4 \alpha} -\frac{1}{8} \right) |\circo h|^4 + \frac{1}{4 \beta^{1/3}}  [\Tr_g(\circo h ^3)]^{4/3} \quad. 
\end{equation}
Then, the  functional 
$$\mc P_{(\alpha, \beta)} (f):= \int_{\M} P_{(\alpha, \beta)} (g,h) d \mu_g  $$
is invariant under conformal transformations of $\R^5$ centered off of $f(\M^4)$. 
\\Moreover, if $3\alpha+\beta \leq 12$, then $P_{(\alpha, \beta)}(g,h)\geq 0$  for every immersed hypersurface $f(\M^4)\subset \R^5$ and $P_{(\alpha, \beta)} (g,h)\equiv 0$ on $E\subset f(\M^4)$ if and only if $H\equiv 0$ on $E$.
\end{proposition}
 
 We now answer the last question 3) by proving that round spheres are the points of strict global minimum for the functional $\mc P_{(\alpha, \beta)}$, for   $\alpha, \beta>0$ with  $\alpha \leq 2$ and $3\alpha+\beta \leq 12$. This result may be seen as the 4-d analogue of the celebrated  2-d   theorem of Willmore \cite{Will} asserting that the  Willmore functional is strictly minimized by embedded round 2-d spheres of $\R^3$.

\begin{theorem} \label{thm:3}
For any smooth isometric  immersion $f:(\M^4,g)\hookrightarrow (\R^5, \delta_{\mu\nu})$ of a closed connected $4$-manifold  $(\M^4,g)$, let 
$$F(g,h):=\detg(h)+Z(\circo h), $$
where $Z(\circo h)$ is a non negative real function of the eigenvalues of $\circo h$, homogeneous of degree $4$ (that is $Z(t \circo h)=t^4 Z(\circo h)$, for every $t \in \R$), such that
\begin{enumerate}
\item $F(g,h)$ is always non-negative and if $F(g,h)=0$ then $H=\frac{1}{4} \Tr_{g}(h)=0$,
\item if  $Z(\circo h)=0$ then $\circo h=0$.
\end{enumerate}

Then the functional
$$\mc F(f):= \int_{\M} \left[ \detg(h)+Z(\circo h) \right] \, d\mu_g $$
is invariant under conformal transformations of $\R^5$ centered off of $f(\M^4)$ and 
\begin{align}\label{eq:LBmcF}
\mc F (f)   \geq \frac{ 8 \pi ^ 2 } { 3}  \quad,
\end{align}
with equality if and only if $f(\M)$ is embedded in $ \R ^ 5 $ as a round sphere. 
\\In particular, this holds for the functional $\mc  P_{(\alpha, \beta)}$ above provided $\alpha, \beta>0$, $\alpha \leq 2$ and $3\alpha+\beta \leq 12$.
\end{theorem}

\begin{proof}
Since by assumption both of the functions $F(g,h)$  and  $Z(\circo h)$ are non-negative, denoting
$$\M^+:=\{x \in \M: \detg(h)\geq 0\} \subset \M \quad,$$
we  have 
\begin{equation}\label{eq:intPab}
\mc F (f)= \int_{\M} F(g,h) d \mu_g  \geq \int_{\M^{+}} F(g,h) d \mu_g =\int_{\M^{+}} \big[  \detg(h)+Z(\circo h) \big] d \mu_g  \geq  \int_{\M^+} \detg(h) \, d \mu_g,
\end{equation}
with equality in the first estimate if and only if  $H\equiv 0$ on $\M\setminus \M^{+}$ and equality in the
second estimate if and only if $\circo h\equiv 0$ on $\M^{+}$. We claim this is possible only if $f(\M)$ is
embedded as a round sphere.

We first show that $\M\setminus \M^{+}=\emptyset$. Indeed, since $\M\setminus \M^{+}\subset \M$ is open, if it was not
empty it would be a piece of a minimal hypersurface. On the other hand, the connected
components of $f(\M^{+})$ are embedded as subsets of either round spheres or affine hyperplanes.
Thus, by the continuity of $H$ and the connectedness of $\M$, only the latter case
can occur. It follows that $f(\M)$ is a closed minimal hypersurface immersed in $\R^{5}$; but
this is a contradiction since in $\R^{n}$ there are no closed minimal immersed hypersurfaces
(one can argue either by maximum principle or via monotonicity formula, the argument
is very classical so we don't repeat it here).
We have just proved that $\M=\M^{+}$. Therefore $f(\M)\subset \R^{5}$ is embedded either as an
affine hyperplane or as a round sphere, and by the same argument above only the latter case
can occur.

We now prove the validity of the lower bound \eqref{eq:LBmcF}.
Observing that $|\detg(h)|=|\detg(D\gamma)|= |J(\gamma)|$, where $\gamma: \M^4 \to S^4$ is the Gauss map of the immersion $f$ and $J(\gamma)$ is its Jacobian, we get
\begin{equation}\label{eq:gammaJ}
\int_{\M^+} \detg(h) \, d \mu_g = \int_{\M^+}  |J(\gamma)| d \mu_g \geq Vol_{S^4}(\gamma(\M^+)) \quad,
\end{equation}
where  $Vol_{S^4}(\gamma(\M^+))$ is the volume of $\gamma(\M^+)\subset S^4$ with respect to the volume form of $S^4$.  But now it is well known that $\gamma(\M^+)=S^4$,  since $f(\M^4)$ is a \emph{closed} hypersurface (the standard argument is to consider, for every $\nu \in S^4$, an affine  hyperplane of $\R^5$ orthogonal to $\nu$ and  very far from $f(\M^4)$;  then one translates such affine hyperplane towards $f(\M)$, keeping the orthogonality with $\nu$,  up to the first tangency point $f(x)$. There,  one has $\gamma(x)=\nu$ and moreover, since by the construction above $f(x)$ is the first tangency point,  $f(\M^4)$ must lie on just one side of its  affine tangent space at $f(x)$, namely such translated hyperplane.
Therefore all the eigenvalues of $g^{-1} h$ at $x$ have the same sign and thus, in particular,  $\detg(h)|_{x}\geq 0$. This proves that $\nu \in \gamma(\M^+)$).    

Combining this last fact with \eqref{eq:intPab} and  \eqref{eq:gammaJ}, we get that 
$$\mc F (f) \geq |S^4| =\frac{8 \pi^2}{3}\quad.  $$
Equality in the last formula of course implies equality in \eqref{eq:intPab} and  \eqref{eq:gammaJ} but, as already observed above, equality in \eqref{eq:intPab} implies that $f(\M^4)$ is embedded as a round sphere in $\R^5$. 

To see that $\mc P_{(\alpha, \beta)}$ satisfies the assumptions, observe first that 1) is ensured by Proposition \ref{prop:2}. Regarding 2) note that  if $\alpha, \beta>0$ and   $\alpha \leq 2$, then
$$
P_{(\alpha, \beta)} (g,h):= \detg(h)+ \frac{1}{4} \Tr_g(\circo h^4)+\left(\frac{1}{4 \alpha} -\frac{1}{8} \right) |\circo h|^4 + \frac{1}{4 \beta^{1/3}}  [\Tr_g(\circo h ^3)]^{4/3} \geq  \detg(h),
$$
with equality if and only if $\Tr_g(\circo h^4)=|\circo h^2|^2 \equiv 0$; but that happens if and only if $\circo h\equiv 0$. 
\end{proof}

By a similar proof one can show the following higher dimensional generalization, however obtaining a comparably explicit formula for $F(g,h)$ is less clear. 

\begin{theorem} \label{thm:4}
For any smooth isometric  immersion $f:(\M^{2n},g)\hookrightarrow (\R^{2n+1}, \delta_{\mu\nu})$ of a closed connected $2n$-manifold  $(\M^{2n},g)$, let 
$$F(g,h):=\detg(h)+Z(\circo h), $$
where $Z(\circo h)$ is a non negative real function of the eigenvalues of $\circo h$, homogenous of degree $2n$ (that is $Z(t \circo h)=t^{2n} Z(\circo h)$, for every $t \in \R$), such that 
\begin{enumerate}
\item $F(g,h)$ is always non-negative and if $F(g,h)=0$ then $H=\frac{1}{2n} \Tr_{g}(h)=0$,
\item if  $Z(\circo h)=0$ then $\circo h=0$.
\end{enumerate}
Then the functional
$$\mc F(f):= \int_{\M} \left[ \detg(h)+Z(\circo h) \right] \, d\mu_g $$
is invariant under conformal transformations of $\R^{2n+1}$ centered off of $f(\M^{2n})$ and 
\begin{align*}
\mc F (f)   \geq \omega_{2n} \quad,
\end{align*}
where $\omega_{2n}$ is the hypersurface area of a unit sphere $ \mbb{S}^{2n}\subset \R ^{2 n+1}$ with equality if and only if $f(\M)$ is embedded in $ \R ^ {2n+1} $ as a round sphere. 
\\There exists $C=C(n)>0$ such that, setting $Z(\circo h)=C \| \circo h \|^{2n}$, the function  $$F(g,h)= \detg(h)+  C \| \circo h \|^{2n}$$
satisfies the assumptions 1) and 2) above.
\end{theorem}

\begin{proof}
The argument for the  first part of the claim is analogous to the proof of Theorem \ref{thm:3}, so we will just show the last statement.
\\Using the orthogonal decomposition for the second fundamental form $ h = \circo h + H g $, clearly one has 
\begin{align*}
\detg  (h)  = \detg \left( \circo h + H  g \right).
\end{align*} 
Now, expanding the determinant, we have that 
\begin{align*}
  \detg\left( \circo h +  H g \right) = \detg (\circo h)  + \sum_{k = 1 } ^ { 2n -2 } H ^ k P_{k}(\circo h ) + H  ^ {2n} \quad,
\end{align*}
where $ P_{k}( \circo h ) $ is a complete contraction of order $ 2n-k$. Note that the term corresponding to $k =2n-1$ does not appear as $\tr_g (\circo h)  = 0 $. We can now apply Young's inequality to conclude
\begin{align*}
\sum_{ k =1 } ^{2n-2} H ^ k P_{ k } (\circo h ) \geq - \e H ^ {2n} - C_\e  \|\circo h \| ^ {2n}.  
\end{align*}
Furthermore, as $ \detg ( \circo h) \geq  -C_{n} \| \circo h \| ^ {2n}$, we find
 \begin{align*}
\detg (h) \geq \left(1 - \e \right) H ^ {2n} - C_{n,\e} \|\circo h \| ^ {2n } .
\end{align*}
Choosing $\e=1/2$ we then get that there exists $C=C(n)$ such that $\detg (h) + C  \|\circo h \| ^ {2n } \geq  \frac{H^{2n}}{2}$, as desired. 
\end{proof}

\end{document}